\documentclass[10pt]{article}

\usepackage{amsmath,amssymb,amsthm,eucal,upref,bm}%

\usepackage{cases}

\usepackage{caption}
\usepackage{subcaption}

\usepackage{tikz}
\usepackage{pgfplots}
\usetikzlibrary{arrows,calc,shapes, decorations.pathreplacing}

\newcommand{\E}{\mathrm{e}\kern0.2pt}%
\newcommand{\D}{\mathrm{d}\kern0.2pt}%
\newcommand{\RR}{\mathrm{I\kern-0.20emR}}

\newtheorem{lemma}{Lemma}[section]
\newtheorem{theorem}{Theorem}[section]
\newtheorem*{definition}{Definition}

\newtheorem{proposition}{Proposition}[section]

\numberwithin{equation}{section}

\newcommand{\banm}{\begin{anm}}
\newcommand{\eanm}{\end{anm}}

\DeclareMathOperator*{\esssup}{ess\,sup}

\begin{document}

\title{{\bf On the Benjamin--Lighthill conjecture \\ for water waves with vorticity}}

\author{Vladimir Kozlov$^a$, Nikolay Kuznetsov$^b$ and Evgeniy Lokharu$^a$}

\date{}

\maketitle

\vspace{-10mm}

\begin{center}
$^a${\it Department of Mathematics, Link\"oping University, S--581 83 Link\"oping}\\
$^b${\it Laboratory for Mathematical Modelling of Wave Phenomena, \\ Institute for
Problems in Mechanical Engineering, \\ Russian Academy of Sciences \\ V.O., Bol'shoy
pr.\ 61, St Petersburg 199178, Russian Federation}

\vspace{1mm}

E-mail: vlkoz@mai.liu.se / V.~Kozlov; nikolay.g.kuznetsov@gmail.com /
N.~Kuznetsov;  evgeniy.lokharu@liu.se / E.~Lokharu \\

\end{center}

\begin{abstract}
We consider the nonlinear problem of steady gravity-driven waves on the free surface
of a two-dimensional flow of an incompressible fluid (say, water). The flow is
assumed to be unidirectional of finite depth and the water motion is supposed to be
rotational. Our aim is to verify the Benjamin--Lighthill conjecture for flows whose
total head (Bernoulli's constant) is close to the critical one; the latter is
determined by the vorticity distribution so that no horizontal shear flows exist
for smaller values of the total head.

Originally, the conjecture was made about irrotational wave trains in order to
describe them in terms of the parameters $Q$ (rate of flow), $R$ (total
head/Bernoulli's constant) and $S$ (flow force). Let $r$ and $s$ be dimensionless
versions of $R$ and $S$, respectively, for fixed $Q$, and let $\cal C$ be the region
in the $(r, s)$-plane whose cusped boundary $\partial \cal C$ represents all
possible uniform streams; moreover, the part of $\partial \cal C$ corresponding to
supercritical streams is included into $\cal C$, whereas the other part not. The
Benjamin--Lighthill conjecture says that (a) each wave train is represented by a
point of $\cal C$ and (b) every point of $\cal C$ corresponds to some wave train. In
2010--11, this form of the conjecture was proved by Kozlov and Kuznetsov for
irrotational waves corresponding to nearcritical values of Bernoulli's constant.

Here, we modify the Benjamin--Lighthill conjecture to adapt it for rotational waves
on unidirectional flows. Let $\omega$ be a vorticity distribution, then the
corresponding cusped region ${\cal C}_\omega$ (its boundary represents all possible
horizontal shear flows) must be truncated by the line $r = r_0$, where the constant
$r_0$ defined by $\omega$ is finite for some vorticity distributions. Under the
assumptions that $\omega$ is Lipschitz continuous and the problem's parameter $r$
attains nearcritical values, we prove the following extended version of the
conjecture. Namely, along with the assertions (a) and (b) formulated above we show
that the correspondence between wave trains and points in ${\cal C}_\omega$ is
one-to-one. Our verification of the conjecture is based on the existence and
uniqueness theorems for the problem with nearcritical values of $r$. These theorems
also yield that there are two different parametrisations for each family of waves
having their crests on a fixed vertical line.

\vspace{1mm}

\end{abstract}

\tableofcontents

\section{Introduction}

We consider the two-dimensional nonlinear problem describing steady waves in a
horizontal open channel of uniform rectangular cross-section. The corresponding
motion of an inviscid incompressible heavy fluid, say water, occupying the channel
is supposed to be rotational with a prescribed vorticity distribution. The reason
for considering this mathematical model is the importance of vorticity for
interaction of waves with currents (see the survey paper \cite{P} by Peregrine)
which commonly occurs in nature as is indicated by observations (see, for example,
\cite{SCJ} and references cited therein). Indeed, an interesting phenomenon
predicted in the framework of this problem is the formation of (possibly multiple)
counter-currents separated one from the other by critical layers (see
\cite{ConstStr1,KK5,Wah} and references cited therein).

However, the aim of the present article is to study the whole set of waves existing
on {\it unidirectional}\/ flows of constant depth that are close to the so-called
critical shear flow. Unlike the irrotational case when the critical uniform stream
is completely defined by its rate of flow and Bernoulli's constant, the vorticity
distribution is essentially involved in the definition of the critical flow in the
rotational case (see Section 2.1). Nevertheless, it occurs that the behaviour of
waves with vorticity in some aspects is similar to that of irrotational ones. In
particular, this concerns the topic of this paper\,---\,the Benjamin--Lighthill
conjecture in the nearcritical regime. For its verification we establish that
amplitudes and slopes of nearcritical waves are small, but the method used for this
purpose is new and essentially differs from that applied in the irrotational case
and based on harmonic analysis. Our approach provides a simple interpretation of
waves existing for every nearcritical value of Bernoulli's constant. Namely, there
are two different parametrisations for each family of waves having their crests on a
fixed vertical line.

\vspace{-2mm}

\subsection{Statement of the Problem}

Let an open channel of uniform rectangular cross-section be bounded below by a
horizontal rigid bottom and let water occupying the channel be bounded above by a
free surface not touching the bottom. In appropriate Cartesian coordinates $(x,y)$,
the bottom coincides with the $x$-axis and gravity acts in the negative
$y$-direction. We use the non-dimensional variables proposed by Keady and Norbury
\cite{KN} (see also Appendix~A in \cite{KK5} for details of scaling); namely,
lengths and velocities are scaled to $(Q^2/g)^{1/3}$ and $(Qg)^{1/3}$ respectively.
Here $Q$ and $g$ are the dimensional quantities for the rate of flow and the gravity
acceleration respectively, whereas $(Q^2/g)^{1/3}$ is the depth of the critical
uniform stream in the irrotational case.

The steady water motion is supposed to be two-dimensional and rotational; the
surface tension is neglected on the free surface of the water, where the pressure is
constant. These assumptions and the fact that water is incompressible allow us to
seek the velocity field in the form $(\psi_y, -\psi_x)$, where $\psi (x,y)$ is
referred to as the {\it stream function}. The vorticity distribution $\omega$ is
supposed to be a prescribed Lipschitz function depending on $\psi$.

We choose the frame of reference so that the velocity field is time-independent as
well as the unknown free-surface profile. The latter is assumed to be the graph of
$y = \eta (x)$, $x \in \Bbb R$, where $\eta$ is a positive continuous function, and
so the longitudinal section of the water domain is $D = \{ x \in \Bbb R , \ 0 < y <
\eta (x) \}$. The following free-boundary problem for $\psi$ and $\eta$ which
describes all kinds of waves has long been known (cf. \cite{KN}):
\begin{eqnarray}
&& \psi_{xx} + \psi_{yy} + \omega (\psi) = 0, \quad (x,y) \in D ; \label{eq:lapp} \\
&& \psi (x,0) = 0, \quad x \in \Bbb R ; \label{eq:bcp} \\ && \psi (x,\eta (x)) = 1,
\quad x \in \Bbb R ; \label{eq:kcp} \\ && |\nabla \psi (x,\eta (x))|^2 + 2 \eta (x) = 3
r, \quad x \in \Bbb R . \label{eq:bep}
\end{eqnarray}
In condition \eqref{eq:bep} (Bernoulli's equation), $r$ is a constant considered as
the problem's parameter and referred to as Bernoulli's constant/the total head. 

Initially, it is natural to impose rather weak assumptions on the unknown functions,
namely, that $\psi \in C^{1}_{loc} (\bar D)$ and $\eta$ is a Lipschitz function on
$\Bbb R$. This allows us to understand the boundary value problem
\eqref{eq:lapp}--\eqref{eq:kcp} in a weak sense. Then the classical Schauder
estimates are applicable because $\omega$ is a Lipschitz function. This implies
that $\psi \in C^{2, \alpha} (D)$ for every $\alpha \in (0,1)$, and so the problem
\eqref{eq:lapp}--\eqref{eq:bep} may be understood in the classical sense. Indeed,
$\nabla \psi$ is continuous up to the boundary which yields that \eqref{eq:bep} is
fulfilled in the classical sense.

Since we are going to study only unidirectional flows, it is assumed that the
horizontal component of the velocity field has the same direction, say to the right,
throughout the flow. This assumption results in the following additional condition:
\begin{equation}
\psi_y (x, y) > 0 \quad \mbox{for all} \ (x, y) \in \bar D . \label{eq:uni}
\end{equation} 

In conclusion of this section, we recall that $(\psi, \eta)$ is called a Stokes-wave
solution of the problem \eqref{eq:lapp}--\eqref{eq:uni} when $\eta$ is a periodic
function with a single crest per wavelength and symmetric about vertical lines going
through crests, whereas $\psi (x, y)$ is a periodic function of $x$ and its period
is the same as that of $\eta$. Furthermore, a non-stream solution $(\psi, \eta)$ is
called a solitary-wave solution if it asymptotes some stream solution as $|x| \to
\infty$ and $\eta$ is symmetric about the vertical line going through the single
crest. The class of stream solutions is analogous to uniform irrotational streams
and is described below in Section~2.1.

\vspace{-2mm}

\subsection{Background} 

To compare the results obtained for rotational and irrotational waves, we recall
what is known in both cases about the {\it whole class of steady}\/ waves.

The first paper concerning this class in the {\it irrotational}\/ case was
\cite{BenL}. In this paper published in 1954, Benjamin and Lighthill conjectured
that the parameters $Q$, $R$, and $S$ `probably determine the wave-train uniquely'.
Here, $Q$ is the volume rate of flow per unit span, $R$ stands for the total head
(Bernoulli's constant), and $S$ is the flow force, and each of these parameters is a
constant of wave motion, that is, it does not depend on the coordinate measured
along the horizontal bottom.

In order to formulate the irrotational Benjamin--Lighthill conjecture in precise
terms we recall that both periodic and solitary waves bifurcate from uniform streams
whose depths are defined by the following equation:
\begin{equation}
\left( \frac{Q}{{\cal H}} \right)^2 + 2 g {\cal H} = 2 R , \quad {\cal H} > 0 ,
\label{dus}
\end{equation}
where $Q$ and $R$ are given and $g$ is the acceleration due to gravity. Namely, let
$Q$ be fixed, then for $R = R_c = \frac{3}{2} (Q g)^{2/3}$ there exists only one
positive root of (\ref{dus})---the double root ${\cal H}_c = (Q^2/g)^{1/3}$. The
corresponding uniform stream is called {\it critical} because for $R < R_c$ there
are no positive roots at all, whereas for $R > R_c$ the equation has two positive
roots ${\cal H}_-$ and ${\cal H}_+$ such that ${\cal H}_- < {\cal H}_c < {\cal
H}_+$. The uniform stream whose depth is equal to ${\cal H}_-$ (${\cal H}_+$) is
called {\it supercritical} ({\it subcritical}, respectively).

Using $R_c$ and $S_c = \frac{3}{2} (Q^4 g)^{1/3}$ (the critical values of $R$ and
$S$, respectively), two non-dimensional characteristics of flows supporting waves
are defined as follows:
\[ r = R / R_c \quad \mbox{and} \quad s = S / S_c .
\]
We recall that $r$ is usually considered as a given parameter in the problem of
steady waves, whereas $s$ depends on its solution as well as on $r$ itself; see, for
example, formula (2.4) in \cite{Ben}. According to it, for $r > 1$ we have the
following. The value of $s$ corresponding to a supercritical stream is equal to $s_-
(r) = (2 + \nu_-) / (3\nu_-^{1/3})$, whereas
\begin{equation}
s_+ (r) = \frac{1}{3 \nu_+^{1/3}} \left[ 2 + \nu_+ + \frac{1}{8} \left( \frac{{\cal H}_+}
{{\cal H}_-} - 1 \right) \left( 3 - \sqrt{1 + 8 \nu_+} \right)^2 \right]
\label{eq:s+}
\end{equation}
corresponds to a subcritical one. Here, $\nu_-  > 1$ $(\nu_+ < 1)$ is the larger
(smaller, respectively) positive root of the following equation:
\[ r = \frac{1 + 2 \nu}{3 \nu^{2/3}} \, . 
\]
It is worth mentioning that $\nu_-$ $(\nu_+)$ is the Froude number squared of
the supercritical (subcritical, respectively) uniform stream. Both $s_- (r)$ and 
$s_+ (r)$ monotonically increase with $r$, and $s_+ (r) > s_- (r)$ for all $r > 1$, 
thus defining the cuspidal region
\[ {\cal C} = \{ (r,s): 1 \leq r, \ s_- (r) \leq s < s_+ (r) \}  \, ,
\]
see Figure 1 (a scaled version of Figure~2 in \cite{BenL}). Here, the pairs $(r,
s_- (r))$ corresponding to supercritical uniform streams are included because they
also represent solitary-wave disturbances of these streams when these disturbances
exist. In terms of ${\cal C}$ the Benjamin--Lighthill conjecture is as follows:

\vspace{1mm}

(a) Every steady wave train is represented by a point in ${\cal C}$.

(b) To every point in ${\cal C}$ corresponds some steady wave train.

(c) The correspondence between wave trains and points in ${\cal C}$ is
one-to-one.

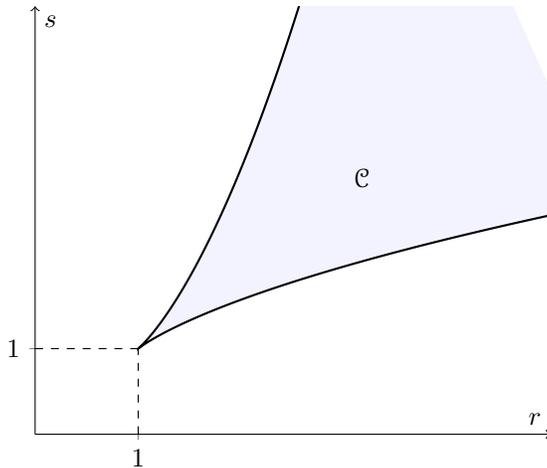
\begin{figure} \label{Pic2}
\centering
\begin{tikzpicture}
\begin{axis}[axis lines = middle, inner axis line style={->}, smooth,xlabel = $r$, 
ylabel =$s$,xmin=0,xmax=5, ymin=0,ymax=5, xtick={0,1},ytick={0,1},
xticklabels={0,$1$},yticklabels={0,$1$}]
\addplot[thick, variable=\t, 
domain=0.2:4,samples=400, fill=blue,fill opacity=0.05] ({1/3*(2/t + t^2)}, 
{1/3*(2*t + 1/t^2)});

\addplot[dashed] coordinates {(0, 1)(1, 1) };
\addplot[dashed] coordinates {(1, 0)(1, 1)};

\addplot [black, nodes near coords=${\cal C}$,every node near
coord/.style={anchor=180}] coordinates {( 3, 3)};

\end{axis}
\end{tikzpicture}
\vspace{-4mm} 
\caption{Possible values of $(r,s)$ representing steady waves
according to \cite{BenL}.} \vspace{-2mm}
\end{figure}

In \cite{BenL} (see also \cite{Ben}), it was demonstrated that this conjecture
is true in the framework of a one-dimensional approximate theory derived under
the assumption that both $r-1$ and $s-1$ are small, that is, waves are long and
of small amplitude.

Since no steady waves other than Stokes and solitary waves had been known when the
conjecture was proposed, first results confirming assertion (a) were obtained for
Stokes waves; namely, Keady and Norbury \cite{KN} (see also \cite{Ben}) proved that
Stokes waves are represented by inner points of ${\cal C}$. On the other hand,
Ovsyannikov \cite{O} supposed that $(r, s_- (r))$ uniquely determines a solitary
wave, but Plotnikov \cite{P} disproved this conjecture. He established that there
are values of $r$ for which at least two geometrically distinct solitary-wave
profiles exist. Thus, assertion (c) of the Benjamin--Lighthill conjecture does not
hold in its general form. It is worth mentioning that as early as 1974 Plotnikov's
result had been found numerically by Longuet-Higgins and Fenton \cite{LHF}.

Further results confirming the conjecture are as follows. Amick and Toland \cite{AT}
showed that $r > 1$ for solitary waves (another proof of this fact was given by
McLeod \cite{Mc}). Subsequently, this fact was proved in \cite{KK0} for all steady
waves irrespective of the type (Stokes, solitary, whatever). On the other hand,
numerical results obtained by Cokelet \cite{Cok} show that assertion (b) is not
true, at least for Stokes waves. There is a bound dividing ${\cal C}$ so that {\it
all}\/ Stokes-wave solutions lie to the left of it. This bound is formed by points
corresponding to waves that have greatest total head and flow force for a given
wavelength; the corresponding line crosses the lower part of $\partial {\cal C}$ at
some distance from the cusp vertex and approaches the curve \eqref{eq:s+} for large
values of $r$ (see Figure~1(a) in \cite{KN}). Other important results were obtained
by Amick and Toland \cite{AT}, who combined the existence theorems for Stokes and
solitary waves with the proof of the convergence of Stokes waves to solitary ones in
the long-wave limit.

In the nearcritical case ($r$ is close to one), the general structure of the whole
set of irrotational waves was obtained in \cite{KK2,KK1}. It is as follows: only
solitary and Stokes waves exist and they are parametrised by the depth at the crest
which varies from the depth of the subcritical uniform flow to that of the solitary
wave at its crest. This implies that assertions (a) and (b) of the conjecture are
true in this case. Moreover, this hierarchy of waves is in agreement with the
one-dimensional approximate theory investigated in \cite{BenL} as well as with the
bound obtained in \cite{Cok}.

Now, we turn to results for waves with vorticity. To the authors' knowledge, so far
only Groves and Wahl\'en \cite{GW} studied small-amplitude Stokes and solitary waves
using a unified approach based on the so-called spatial dynamics in which the
horizontal coordinate plays the role of time. Their analysis demonstrates the
existence of a continuous branch of wave solutions that consists of a solitary wave
of elevation and a family of Stokes waves.

\vspace{-2mm}

\section{Main Results}

The exact formulation of the Benjamin--Lighthill conjecture for waves with vorticity
includes the same assertions (a)--(c) as in the irrotational case. However, the
vorticity distribution must be involved in the definition of the corresponding
cuspidal region (hence this region will be denoted by ${\cal C}_\omega$ in what
follows) along with the total head $r$ and the flow force invariant $s$. The latter
is introduced below following the paper \cite{KN} by Keady and Norbury. Similarly to
the irrotational case, we define ${\cal C}_\omega$ in terms of unidirectional
horizontal shear flows which are analogous to uniform streams. Besides, {\it stream
solutions} describing these flows have more complicated structure. For example,
there are several definitions of the so-called Froude number for shear flows neither
of which is universal like that for the irrotational uniform streams (see \cite{W2},
p.~95, where these definitions are listed). Therefore, prior to the formulation of
main results we outline basic properties of stream solutions and provide a
classification of vorticity distributions based on these properties. This
classification is also important when defining the cuspidal region ${\cal
C}_\omega$.

\vspace{-2mm}

\subsection{Stream Solutions}

By a {\it stream} (shear-flow) solution we mean a pair $(u (y), \, d)$; $u$ stands
for the stream function instead of $\psi$ and the constant depth of flow $d$
replaces the wave profile $\eta$. Then problem \eqref{eq:lapp}--\eqref{eq:bep}
reduces the following one:
\begin{equation}
u'' + \omega (u) = 0 \ \ \mbox{on} \ (0, d) , \ \ \ u (0) = 0, \ \ \ u (d) = 1 , \ \
\ |u' (d)|^2 + 2 \, d = 3 \, r . \label{eq:ss1}
\end{equation}
Here, the prime symbol denotes differentiation with respect to $y$. A detailed study
of these solutions including those that describe flows with counter-currents is
given in \cite{KK}. In particular, it is shown that the set of unidirectional
solutions of the first three relations \eqref{eq:ss1} is parametrised by $\lambda =
u' (0)$ which satisfies the inequality
\[ \lambda \geq \lambda_0 = \sqrt{2 \max_{0 \leq \tau \leq 1} \Omega (\tau)}, \quad 
\mbox{where} \ \Omega(\tau) = \int_0^\tau \omega(t) \D t .
\]
This is a consequence of the following expressions for $u$ and $d$ (implicit and
explicit respectively):
\begin{equation}
y = \int_0^u \frac{\D \tau}{\sqrt{\lambda^2 - 2 \, \Omega (\tau)}}  \ \ \mbox{and} \
\ d = \int_0^1 \frac{\D \tau}{\sqrt{\lambda^2 - 2 \, \Omega (\tau)}} \, .
\label{eq:d}
\end{equation}
The function $d \ [= d (\lambda)]$ decreases strictly monotonically and tends to
zero as $\lambda \to +\infty$ (see Figure 2), whereas $d_0 = \lim_{s \to \lambda_0 +
0} d (\lambda)$ can be finite or infinite depending on the behaviour of $\Omega$ on
$[0, 1]$ (see below). It should be noted that the solution \eqref{eq:d} is well
defined for $\lambda = \lambda_0$ when $d_0 < +\infty$.

\begin{figure}[!ht]
\begin{minipage}{.49\textwidth}
\begin{tikzpicture}[yscale=1.50]
\draw[->] (0,0) -- (0,2.5) node[left]{$d$};
\draw[->] (0,0) -- (4.2,0) node[right]{$\lambda$};
\draw [thick] (0.5, 1.5) to[out=-90,in=178] (4,0.05);
\draw[-,dashed] (0.5,0) node[below]{$\lambda_0$} -- (0.5,1.5);
\draw[-,dashed] (0,1.5) node[left]{$d_0$} -- (0.5,1.5);
\draw[-,dashed] (1.5,0) node[below]{$\lambda_c$} -- (1.5,0.28);
\draw[-,dashed] (0,0.28) node[left]{$d_c$} -- (1.5,0.28);
\end{tikzpicture}
\vspace{-4mm}
\caption{A sketch of the graph \\ of $d (\lambda)$ in the case when $d_0 < \infty$.}
\end{minipage}
\begin{minipage}{.49\textwidth}
\begin{tikzpicture}[scale=1.50]
\draw[->] (0,0) -- (0,2.5) node[left]{$r$};
\draw[->] (0,0) -- (3,0) node[right]{$\lambda$};
\draw [thick] (0.5, 1.5) to[out=-95,in=180] (1,0.5) to[out=5,in=-120] (2.9,2.2);
\draw[-,dashed] (0.48,0) node[below]{$\lambda_0$} -- (0.48,1.5);
\draw[-,dashed] (1,0) node[below]{$\lambda_c$} -- (1,0.5);
\draw[-,dashed] (0,0.5) node[left]{$r_c$} -- (1,0.5);
\draw[-,dashed] (0,1.5) node[left]{$r_0$} -- (0.5,1.5);
\end{tikzpicture}
\vspace{-4mm} 
\caption{A sketch of the graph \\ of ${\cal R} (\lambda)$ in the
case when $d_0 < \infty$.}
\end{minipage}
\end{figure}

Furthermore, according to the last relation \eqref{eq:ss1}, the value of $\lambda$
in formulae \eqref{eq:d} must be determined from the equation
\begin{equation}
r = {\cal R} (\lambda) , \quad \mbox{where} \ {\cal R} (\lambda) = [ \lambda^2 - 2 \, 
\Omega (1) + 2 \, d (\lambda) ] / 3 . \label{eq:calR}
\end{equation}
One comes to this conclusion by substituting the first expression \eqref{eq:d} into
the last relation \eqref{eq:ss1}. It is easy to check that the function ${\cal R}
(\lambda)$ has only one minimum, say $r_c > 0$ attained at some $\lambda_c >
\lambda_0$ (see Figure 3). Hence $r_c$ is the critical value of $r$ in the same
sense as $r = 1$ is critical in the irrotational case, that is, no $\lambda$ can be
found from \eqref{eq:calR} when $\lambda < \lambda_c$. It is clear that
\[ \int_0^1 \frac{\D \tau}{[ \lambda_c^2 - 2 \Omega (\tau) ]^{3/2}} = 1 ,
\]
which, in particular, implies that $\lambda_c^2 - \lambda_0^2 \leq 1$. 

If $d_0 = +\infty$, then for every $r > r_c$ the equation \eqref{eq:calR} has two
solutions $\lambda_+ (r)$ and $\lambda_- (r)$ such that $\lambda_0 < \lambda_+ <
\lambda_c < \lambda_-$. By substituting $\lambda_+$ and $\lambda_-$ into
\eqref{eq:d}, one obtains two stream solutions, say $(u_+, d_+)$ and $(u_-, d_-)$.
However, if $d_0 < +\infty$, then both $\lambda_+$ and $\lambda_-$, and consequently
the corresponding stream solutions satisfying inequality \eqref{eq:uni} exist only
for $r \in (r_c, r_0)$, where $r_0 = {\cal R} (\lambda_0)$. It should be noted that
$d_0 > d_+ > d_c > d_-$, where $d_c = d (\lambda_c)$. The shear flows described by
$(u_+, d_+)$ and $(u_-, d_-)$ are analogous to the uniform sub- and supercritical
flows respectively existing in the irrotational case.

It is worth mentioning that the values $d_+ (r)$ and $d_- (r)$ provide important
bounds for wave profiles on unidirectional rotational flows (see \cite{KKL2} for the
proof). Namely, if $r \in (r_c, r_0)$, then the inequalities
\[ d_-  (r) < \eta (x) \ \ \mbox{for all} \ x \in \Bbb R \quad \mbox{and} \quad 
\inf \eta < d_+ (r) \leq \sup \eta 
\]
hold for all non-stream solutions. Moreover, the last inequality is strict provided
$\sup \eta$ is attained somewhere. Finally, the problem
\eqref{eq:lapp}--\eqref{eq:bep} has no solution at all if $r < r_c$ and only the
stream solution exists when $r = r_c$.

To analyse the dependence of $d_0$ on the vorticity distribution the following
three options were considered in \cite{KK}:

\begin{center}
\begin{tabular}{ | l |  p{5cm} |} 
\hline (i) & $d_0 = +\infty$ \\
\hline (ii) & $d_0< +\infty, \ u' (0) = 0, \ u'(d_0) \neq 0$ \\
\hline (iii) & $d_0< +\infty, \ u'(d_0) = 0$ \\ \hline
\end{tabular} 
\end{center}

\noindent Thus, $d_0 < +\infty$ if either $u'(0) = 0$ or $u'(d_0) = 0$, and this \
classification can be reformulated in terms of the vorticity distribution as follows:

\vspace{1mm}

\noindent \ \ (i) $\max_{0 \leq p \leq 1} \Omega (p)$ is attained either at an
inner point of $(0, 1)$ or at one (or both) of the end-points. In the latter
case, either $\omega (1) = 0$ when $\Omega (1) > \Omega (p)$ for $p \in (0, 1)$
or $\omega (0) = 0$ when $\Omega (0) > \Omega (p)$ for $p \in (0, 1)$ (or both
of these conditions hold simultaneously).

\noindent \ (ii) $\Omega (0) > \Omega (p)$ for $p \in (0, 1]$ and $\omega (0) <
0$.

\noindent (iii) $\Omega (p) < \Omega (1)$ for $p \in (0, 1)$ and $\omega (1) >
0$. Moreover, if $\Omega (1) = 0$, then $\omega (0) < 0$ and $\omega (1) > 0$ must
hold simultaneously.

\vspace{1mm}

\noindent Conditions (i)--(iii) define three disjoint sets of vorticity
distributions whose union gives the whole set of distributions that are continuous
on $[0,1]$.

\vspace{-2mm}

\subsection{Flow Force and the Cuspidal Region ${\cal C}_\omega$}
 
To the authors' knowledge, the flow force invariant $s$ for rotational waves was
introduced by Keady and Norbury \cite{KN}; up to a slight difference in the
definition of $\Omega$, their definition is as follows. Let $(\psi, \eta)$ be a
solution of the problem \eqref{eq:lapp}--\eqref{eq:bep}, then
\begin{eqnarray}
&& \!\!\!\!\!\!\!\!\!\!\!\! s = s (\psi, \eta) = \left[ r + \frac{2}{3} \Omega (1)
\right] \eta (x) \nonumber \\ && - \frac{1}{3} \Big\{ \eta^2 (x) - \int_0^{\eta (x)}
\left[ \psi_y^2 - \psi_x^2 - 2 \Omega (\psi) \right] \, \D y \Big\} . \label{eq:s}
\end{eqnarray}
Indeed, using relations \eqref{eq:lapp}--\eqref{eq:bep}, it is straightforward to
check that this expression is independent of $x$.

By $s_\pm (r)$ we denote the flow force corresponding to the stream solution 
$(u_\pm (y), d_\pm)$, and so
\[ s_\pm (r) =  \left[ r + \frac{2}{3} \Omega (1) \right] d_\pm - \frac{1}{3} \Big\{
d_\pm^2 - \int_0^{d_\pm} \left[ (u_\pm)_y^2 - 2 \Omega (u_\pm) \right] \, \D
y \Big\} . 
\]
It should be noted that the equation $s = s_+ (r)$ $(s = s_- (r))$ gives the upper
(lower respectively) curve bounding the cuspidal region ${\cal C}_\omega$.

It occurs that the curve $s = s_- (r)$ always goes to infinity being defined for all
$r \in [r_c , +\infty)$, whereas $s = s_+ (r)$ goes to infinity only when $r_0 =
+\infty$ or, what is the same, conditions (i) hold for the vorticity distribution.
Thus, the cuspidal region ${\cal C}_\omega$ is similar to ${\cal C}$ in this case.
On the contrary, if $\omega$ satisfies either of conditions (ii) and (iii), then
both $d_0$ and $r_0$ are finite, and so the curve $s = s_+ (r)$ terminates at the
point $(r_0, s_+ (r_0))$. Hence the region ${\cal C}_\omega$ is bounded under these
conditions and $\partial {\cal C}_\omega$ consists of two arcs and the segment that
connects them and lies on the line $r = r_0$.

\vspace{-2mm}

\subsection{Definitions and Formulations}

Since we are going to study waves only for nearcritical values of $r$, it is
convenient to suppose that
\begin{equation} \label{BdR}
r \leq {\cal R} ( [\lambda_c + \lambda_0] / 2 ) .
\end{equation}
Also, the following notation will be used below:
\begin{equation} \label{omegaL}
\omega_0 = \max_{[0,1]} |\omega| , \quad \omega_1 = \omega_0 + \esssup_{[0,1]} |\omega'| .
\end{equation}

Now, we specify the problem to be investigated in this paper.

\begin{definition} {\rm Let $\omega$ be a Lipschitz vorticity distribution. We say 
that $(\psi, \eta)$ belonging to $C^1_{loc}(\bar D) \times W^{1,\infty} (\Bbb R)$ is
a solution of problem $\rm P_r^M$ for some $M > 0$ and $r > r_c$ ($r_c > 0$ is the
critical value of Bernoulli's constant for $\omega$), if the following conditions
are fulfilled. The inequality $|\eta'(x)| \leq M$ is true a.e. on $\Bbb R$ and
\eqref{eq:uni} holds for $\psi$. The latter function satisfies the equation
\eqref{eq:lapp} in a weak sense, whereas the boundary conditions
\eqref{eq:bcp}--\eqref{eq:bep} are valid pointwise.}
\end{definition}

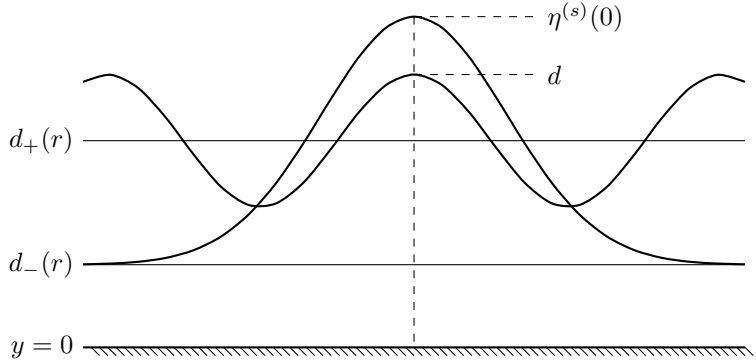
\begin{figure} \label{Pic5}
\centering
\begin{tikzpicture}[scale=1.1]
\draw[-] (-4,1.5) node[left]{$d_-(r)$} -- (4,1.5);
\draw[-] (-4,3) node[left]{$d_+(r)$} -- (4,3);

\draw[-, thick] (-4,0.5) node[left]{$y=0$} -- (4,0.5);

\foreach \x in {1,2,...,80}{
\draw[-] (-4 + 0.1*\x,0.5) -- (-4 + 0.1*\x + 0.1,0.4);
}

\draw[-,dashed] (0,4.5) -- (0,0.5);

\draw[-,dashed] (0,4.5) -- (1.5,4.5) node[right]{$\eta^{(s)}(0)$};

\draw[-,dashed] (0,3.8) -- (1.5,3.8) node[right]{$d$};

\draw[thick] plot[domain=-4:4,smooth] (\x,{1.5+3*exp(-0.4*\x*\x)});

\draw[thick] plot[domain=-4:4,smooth] (\x,{3+0.8*cos(1.69*\x r)});

\end{tikzpicture}
\caption{A sketch of the hierarchy of waves corresponding to a nearcritical
value of $r$.} \vspace{-2mm}
\end{figure}

Let us turn to the formulation of main results concerning problem $\rm P_r^M$.
Our first theorem provides a complete description of the set of waves existing
in the nearcritical regime; these are the family of Stokes waves and a solitary
wave naturally parametrised by their heights at the crest (see Figure~4).

\begin{theorem} \label{UniqThm} 
For any $M > 0$ there exists $r' \in (r_c, r_0)$ (it also depends on $\omega_1$)
such that the following assertions are true:
\begin{itemize}
\item [\sc (I)] For every $r \! \in \! (r_c, r']$ problem ${\rm P^M_r}$ has one and
only one solitary-wave solution $(\psi^{(s)}, \eta^{(s)})$ such that $\eta^{(s)}$
attains its maximum at $x = 0$ and is an even function. All other solitary-wave
solutions are horizontal translations of $(\psi^{(s)}, \eta^{(s)})$. 
\item [\sc
(II)] If $(\psi, \eta)$ is a Stokes-wave solution of problem ${\rm P^M_r}$ with $r
\in (r_c, r']$, then the following inequalities hold:
\begin{equation} \label{t2:II}
d_+ = d (\lambda_+ (r)) < \max_{x \in \Bbb R} \eta (x) < \eta^{(s)} (0) ,
\end{equation}
where $d (\lambda)$ is defined by the second formula \eqref{eq:d} and $\lambda_+
(r)$ is the smallest root of \eqref{eq:calR}. 
\item[\sc (III)] For every $r \!
\in \! (r_c, r']$ and every $d \in ( d_+, \eta^{(s)} (0) )$ problem ${\rm
P^M_r}$ has one and only one Stokes-wave solution $(\psi, \eta)$ such that
\[ \max_{ x \in \Bbb R} \eta (x) = \eta (0) = d .
\]
Thus, the family of Stokes waves with crests on the $y$-axis is parametrised by $d$.
All other Stokes-wave solutions are horizontal translations of $(\psi, \eta)$.
\item[\sc (IV)] For every $r \in (r_c, r']$ all solutions of problem ${\rm P^M_r}$
are exhausted by those described in assertions {\sc (I)} and {\sc (III)} and stream
solutions.
\end{itemize}
\end{theorem}

It should be mentioned that the left-hand side of inequality \eqref{t2:II} was
obtained in \cite{KKL2}; here we include it for the sake of completeness. The next
assertion shows that all waves considered in Theorem~\ref{StabThm} depend on the
depth at the crest continuously, but the continuity is understood in a certain
integral sense.

\begin{theorem} \label{StabThm}
Given any $M > 0$, then there exists $r'' \in (r_c, r_0)$ (it also depends on
$\omega_1$) such that for any two solutions of problem $\rm P_r^M$ with $r \in (r_c,
r'']$, say $(\psi^{(1)}, \eta^{(1)})$ and $(\psi^{(2)}, \eta^{(2)})$, the following
inequality
\begin{eqnarray*}
&& \int_{\RR} |\eta^{(1)} (x) - \eta^{(2)} (x)|^2 \, \E^{-\theta |x-x_0|} \D x \\ &&
\ \ \ \ \ \ \ \ \leq C \left[ | \eta^{(1)} (x_0) - \eta^{(2)} (x_0) |^2 + |
\eta_x^{(1)} (x_0) - \eta_x^{(2)} (x_0) |^2 \right]  
\end{eqnarray*}
holds for any $x_0 \in \Bbb R$. The constants $C$ and $\theta$ depend only $\omega_0$.
\end{theorem}

Theorems \ref{UniqThm} and \ref{StabThm} are analogous to Theorems 2.1 and 2.2,
respectively, proved in \cite{KK2}. In the latter paper, the mentioned theorems
provided the basis for verification of assertions (a) and (b) of the
Benjamin--Lighthill conjecture for irrotational waves. Here, Theorems \ref{UniqThm}
and \ref{StabThm} serve for the same purpose for waves with vorticity (see Section 7
below), but our techniques used for establishing these theorems differ radically
from those applied in \cite{KK2}.

Let $r' \in (r_c, r_0)$ be the number whose existence is established in Theorem
\ref{UniqThm}, and let $r \in (r_c, r']$. By  ${\cal P}_{\rm r}^{\rm M}$ we
denote the set of all solutions to problem $\rm P_r^M$ that have the following
property. The second component $\eta$ has a crest on the $y$-axis (see
assertions {\sc (I)} and {\sc (III)} of Theorem~\ref{UniqThm} and Figure~4).
Then formula \eqref{eq:s} defines the following map:
\begin{equation}
{\cal P}_{\rm r}^{\rm M} \ni (\psi, \eta) \mapsto s \in [ s_- (r), s_+ (r) ) .
\label{eq:maps}
\end{equation}
The last main result concerns this map; namely, the set ${\cal P}_{\rm r}^{\rm M}$
is parametrised by the flow force $s$ which is an alternative parametrisation to
that considered in Theorem \ref{UniqThm}.

\begin{theorem} \label{BLThm} 
For any $M > 0$ there exists $r''' \in (r_c, r_0)$ (it also depends on $\omega_1$)
such that the map \eqref{eq:maps} is one-to-one for any $r \in (r_c, r''']$.
\end{theorem}

This theorem yields that along with the parametrisation described in assertion {\sc
(III)} of Theorem~\ref{UniqThm} there is another parametrisation for the family of
waves having their crests on the $y$-axis, namely, that in terms of $s$. Theorem
\ref{BLThm} also allows us to verify assertion (c) of the Benjamin--Lighthill
conjecture for rotational waves; irrotational ones are included as a particular
case with zero vorticity.

It follows from our proofs of Theorems \ref{StabThm} and \ref{BLThm} that $r''
\geq r' = r'''$.

\vspace{-2mm}

\section{Reformulation of the Problem and \\ an Auxiliary Sturm--Liouville Problem}

In our proofs, an equivalent problem in a fixed strip $S = \Bbb R \times (0,1)$ is
used instead of the original problem \eqref{eq:lapp}--\eqref{eq:bep} in the unknown
domain $D$. The reformulation is possible for unidirectional flows and is based on
the change of variables referred to as the partial hodograph transform (see Section
3.1) proposed by Dubreil-Jacotin \cite{DJ} in 1934. In Sections 3.2 and 3.3, we
investigate an auxiliary problem related to the linearized reformulated problem.
This Sturm--Liouville problem was considered, for example, in \cite{ConstStr1}, but
the estimates of its eigenvalues and eigenfunctions (see Proposition \ref{PropK1} in
Section 3.3), that are necessary for proving our main results, were not obtained so
far.

\vspace{-3.4mm}

\subsection{Partial Hodograph Transform}

In view of the boundary conditions \eqref{eq:bcp}, \eqref{eq:kcp} and the inequality
\eqref{eq:uni}, putting $q = x$  and $p = \psi (x,y)$, we define a mapping
\[ D \ni (x,y) \mapsto (p,q) \in S = \Bbb R \times (0,1) .
\]
Let us treat the pair $(q, p)$ as independent variables in $S$ and consider $y$ as
the new unknown function for which $h (q, p)$ is the standard notation. A
straightforward calculation shows that
\[ h_q = - \frac{\psi_x}{\psi_y} \quad \mbox{and} \quad h_p = \frac{1}{\psi_y} ,
\]
thus yielding that the problem \eqref{eq:lapp}--\eqref{eq:bep} takes the following
form
\begin{eqnarray}
&& \left( \frac{h_q}{h_p} \right)_q - \frac{1}{2} \left( \frac{1+h_q^2}{h_p^2}
\right)_p - \omega(p) = 0, \quad (q, p) \in S; \label{eq:hdf1} \\ && h(0,q) =
0,\quad q \in \Bbb R ; \label{eq:hdf2} \\ && \frac{1 + h_q^2}{h_p^2} = 3r - 2h,
\quad p = 1, \ q \in \Bbb R  \label{eq:hdf3}
\end{eqnarray}
in terms of new variables. On the other hand, using the formulae
\begin{equation} \label{PsiHrel}
\psi_x = - \frac{h_q}{h_p}  \quad \mbox{and} \quad \psi_y = \frac{1}{h_p} ,
\end{equation}
one recovers the gradient of $\psi$; that is, the velocity field $(\psi_y, -\psi_x)$
in $\bar D$. Besides, the equality $\eta (x) = h (x, 1)$, $x \in \Bbb R$ gives the
free surface profile.

Let us emphasise two main advantages of the relations \eqref{eq:hdf1}--\eqref{eq:hdf3} 
comparing with \eqref{eq:lapp}--\eqref{eq:bep} (to which they are equivalent under
reasonable smoothness assumptions; see \cite{ConstStr1} for a detailed account).
There is only one unknown function $h$ and it is defined on the fixed strip $\bar S$.

In terms of new variables, a stream solution is a single function, say $H
(p,\lambda)$, which is as follows:
\begin{equation} \label{Stream}
H (p, \lambda) = \int_0^p \frac{\D \tau}{\sqrt{\lambda^2 - 2  \, \Omega (\tau)}} .
\end{equation}
Here $\lambda \in (\lambda_c , \lambda_0]$ is the same parameter as in \eqref{eq:d}
and \eqref{eq:calR}. It is clear that $d_+ (r) = H (1, \lambda_+ (r))$.

\subsection{Auxiliary Sturm--Liouville Problem}

It is straightforward to check that the equation obtained by linearization of
\eqref{eq:hdf1} near the stream solution $H$ involves the operator $- \partial_p
\left( H_p^{-3} \partial_p \right)$. This is the reason to consider the following
Sturm--Liouville problem:
\begin{eqnarray}
- \left( \frac{\phi_p}{H_p^3} \right)_p \!\!\!\! &=& \!\!\!\! \mu \frac{\phi}{H_p} ,
\quad p \in (0,1) ; \label{SL1} \\ \phi_p(1) \!\!\!\! &=& \!\!\!\! H_p^3(1) \phi(1)
; \label{SL2}  \\ \phi(0) \!\!\!\! &=& \!\!\!\! 0 . \label{SL3}
\end{eqnarray}
Here $\mu$ is the spectral parameter, but one has to keep in mind that the
eigenvalues and eigenfunctions of this problem depend also on $\lambda$ because $H$
depends on it. In our investigation of the problem \eqref{eq:hdf1}--\eqref{eq:hdf3},
the problem \eqref{SL1}--\eqref{SL3} plays an essential role. Let us denote by $\{
\mu_k \}_{k=0}^{+\infty}$ the sequence of its eigenvalues all of which are simple,
that is,
\[ \mu_0 < \mu_1 < \dots < \mu_n < \dots ;
\]
here $\mu_0$ may be either positive or negative.

\subsubsection{A Preliminary Estimate of $\mu_1$}

Let 
\[\frak M = \max_{p \in [0,1]} H_p = (\lambda^2 - \lambda_0^2)^{-1/2} \ \mbox{and}
\ \frak m = \min_{p \in [0,1]} H_p = \big[ \lambda^2 - 2 \min_{\tau \in [0,1]}
\Omega (\tau) \big]^{-1/2} , 
\]
then
\begin{equation} \label{KK10}
\mu_1 \geq \pi^2 \frac{\frak m}{\frak M^3} .
\end{equation}
Indeed, it is well known that the smallest eigenvalue of the operator $- \D^2 / \D p^2$ on
$(0, 1)$ with the Dirichlet boundary conditions is
\[ \inf_{\phi \in W^{1,2}_0 (0,1)} \frac{\int_0^1 \phi_p^2 \D p}{\int_0^1 \phi^2 \D p}
= \pi^2 .
\]
Since the fundamental eigenfunction $\phi_0$ of the problem \eqref{SL1}--\eqref{SL3}
is non-negative irrespective of the sign of the corresponding eigenvalue $\mu_0$,
the eigenfunction $\phi_1$ has exactly one zero, say $p_* \in (0,1)$. Therefore, we
have
\[ \pi^2 \leq p_*^2 \frac{\int_0^{p_*} \phi_{1 p}^2 \D p}{\int_0^{p_*} \phi_{1 p}^2 
\D p} \leq \frac{\frak M^3}{\frak m} \frac{\int_0^{p_*} \phi_{1 p}^2 H_p^{-3} \D
p}{\int_0^{p_*} \phi_1^2 H_p^{-1} \D p} \leq \frac{\frak M^3}{\frak m} \mu_1 ,
\]
which yields inequality \eqref{KK10}.

\subsubsection{Properties of $\mu_0$ and $\phi_0$}

Let us consider the initial value problem
\begin{equation} \label{Vdisp}
- (H_p^{-3} V_p)_p = \mu H_p^{-1} V , \quad V (0) = 0 , \quad V_p (0) = 1 ,
\end{equation}
depending on the parameters $\mu$ and $\lambda$ (the latter is involved through
$H$), and so we will write $V (p; \lambda, \mu)$ when necessary. It should be noted
that $V$ is a monotone function of $p \in [0,1]$, and it analytically depends on
$\lambda > \lambda_0$ and $\mu \in \Bbb R$. In terms of $V$, the eigenvalue $\mu_0 =
\mu_0 (\lambda)$ of the problem (\ref{SL1})--(\ref{SL3}) is equal to the least root
of the equation $\sigma (\lambda,\mu) = 0$, where
\[ \sigma (\lambda,\mu) = [ H_p (1, \lambda) ]^{-3} V_p (1;\lambda,\mu) - 
V (1;\lambda,\mu) ,
\]
which depends on $\lambda > \lambda_0$ and $\mu \in \Bbb R$ analytically. Moreover,
the eigenfunction $\phi_0$ corresponding to $\mu_0$ is equal to $V$ solving
\eqref{Vdisp} with $\mu = \mu_0$.

In order to find the derivative $\sigma_\mu$, it is convenient to differentiate the
first equality \eqref{Vdisp} with respect to $\mu$, then multiply the result by $V$
and integrate over $(0,1)$. Then we obtain after integration by parts
\begin{equation} \label{KK4}
\sigma_\mu (\lambda, \mu) = \frac{- 1}{V(1; \lambda,\mu)} \int_0^1 \frac{V^2 (p;
\lambda,\mu)} {H_p (p, \lambda)} \D p \, .
\end{equation}
Hence $\sigma_\mu$ is always negative which implies that $\sigma (\lambda,\mu) > 0$
when $\mu < \mu_0 (\lambda)$.

Furthermore, if $\mu = 0$, then $V (p;\lambda,0) = - \lambda^2 \partial_{\lambda} 
H (p; \lambda)$, and so
\begin{equation} \label{KK6}
\sigma (\lambda,0) = \lambda^2 \partial_{\lambda} \Big[ \frac{1}{2 H_p^2 (1;\lambda)}
+ H (1;\lambda) \Big] = \frac{3\lambda^2}{2} {\cal R}' (\lambda) ,
\end{equation}
where the last equality is a consequence of the definitions of $H$ and ${\cal R}$
(see \eqref{Stream} and \eqref{eq:calR} respectively). Since ${\cal R}' (\lambda)$
is negative for $\lambda < \lambda_c$ and positive for $\lambda > \lambda_c$ (see
Figure 2), it follows from \eqref{KK6} and \eqref{KK4} that $\mu_0 (\lambda)$ is
positive (negative) for $\lambda > \lambda_c$ ($\lambda < \lambda_c$ respectively)
and $\mu_0 (\lambda_c) = 0$.

Let us estimate $|\mu_0 (\lambda)|$ for $\lambda < \lambda_c$ or, what is the same,
$\mu_0 (\lambda) < 0$. We have
\[ \mu_0 = \min_{v \in W^{1,2}_0 (0,1)} \frac{\int_0^1 v^2 H_p^{-3} \D p - 
v^2 (1)}{\int_0^1 v^2 H_p^{-1} \D p} \geq \min_{v \in W^{1,2}_0 (0,1)} \frac{\frak
M^{-3} \int_0^1 v^2 \D p - v^2 (1)}{\frak m^{-1} \int_0^1 v^2 \D p} \, ,
\]
where $\frak M$ and $\frak m$ are the same as in \eqref{KK10}. To find the last
minimum one has to find the fundamental eigenvalue of the problem that has the same
form as (\ref{SL1})--(\ref{SL3}), but its constant coefficients are obtained by
changing $H_p$ to $\frak M$. Then the minimum is equal to $- \kappa^2 \frak m \frak M^{-3}$,
where $\kappa = \kappa (\frak M)$ is the root of the following equation:
\[ \frac{\sinh \kappa}{\kappa \cosh \kappa} = \frac{1}{\frak M^3} \, .
\]
Since $\frak M^3 \geq \int_0^1 H_p^{3} \D p > 1$, this equation is uniquely solvable
and we arrive at the following estimate:
\begin{equation} \label{KK9}
|\mu_0| \leq \kappa^2 \frak M^{-2} .
\end{equation}
Using the variational formulation of the problem \eqref{eq:hdf1}--\eqref{eq:hdf3},
one obtains that the right-hand side of the last inequality is a monotone function
of $\frak M$.

Now, we turn to estimates of $V$ and $V_p$. Integrating the equation (\ref{Vdisp})
and using the second boundary condition, we get that
\begin{equation}\label{KK1}
V_p (p) = \frac{H_p^3 (p)}{\lambda^3} + |\mu| H_p^3 (p) \int_0^p \frac{V (p')}{H_p
(p')} \D p' ,
\end{equation}
which gives the following lower estimates:
\begin{equation} \label{KK7}
V (p) \geq \frac{p \, \frak m^3}{\lambda^3} , \quad V_p (p) \geq \frac{\frak m^3}{\lambda^3} .
\end{equation}
In order to obtain an upper estimate for $V$, we integrate \eqref{KK1}, thus
reducing it to an integral equation with a positive operator. Its upper solution
gives an estimate of $V$, and one obtains such a solution from the problem
(\ref{SL1})--(\ref{SL3}) with $H_p^{-3}$ and $H_p^{-1}$ changed to $\frak M^{-3}$
and $\frak m^{-1}$, respectively. This leads to the following inequality:
\begin{equation} \label{KK8}
V \leq \frac{\sinh (\theta p)}{\theta} , \quad \mbox{where} \ \theta^2 = |\mu| 
\frac{\frak M^3}{\frak m} .
\end{equation}
Combining this and \eqref{KK1}, we get that
\begin{equation}\label{KK8a}
V_p \leq \frac{\frak M^3}{\lambda^3} + |\mu| \frac{\frak M^3}{\frak m} \int_0^p 
\frac{\sinh (\theta p)}{\theta} \D p .
\end{equation}

Let us estimate $\mu_0' (\lambda)$, for which purpose we differentiate the equation
\eqref{Vdisp} with respect to $\lambda$. The result multiplied by $V$ we integrate
over $(0,1)$ and after integration by parts arrive at
\begin{equation} \label{KK5}
\sigma_{\lambda} (\lambda,\mu) = \frac{1}{V (1;\lambda,\mu)} \Big( \int_0^1V_p^2
\partial_{\lambda} H_p^{-3} \D p - \mu \int_0^1V^2 \partial_\lambda H_p^{-1} \D p
\Big) .
\end{equation}
Since $\mu_0' (\lambda) = - \sigma_{\lambda} (\lambda,\mu (\lambda)) / \sigma_\mu
(\lambda,\mu (\lambda))$, relations \eqref{KK4} and \eqref{KK5} yield that $\mu
(\lambda)$ is a Lipschitz function on every interval separated from $\lambda_0$.

\subsection{Estimates for Nearcritical Values of $r$}

In this section, we refine the estimates obtained in Sections 3.1 and 3.2 by
applying the assumption \eqref{KK8}, which means that $\lambda \in [
(\lambda_0+\lambda_c)/2, \lambda_c )$.

\begin{proposition} \label{PropK1} 
If $\lambda \in [ (\lambda_0+\lambda_c)/2, \lambda_c)$, then the following
inequalities are valid:

{\rm (i)} $\mu_0 (\lambda) < 0$ and $|\mu_0 (\lambda)| \leq C_1 |\lambda - \lambda_c|$,

{\rm (ii)} $\mu_1 (\lambda) \geq C_2$,

{\rm (iii)}  $C_3 p \leq V(p) \leq C_4 p$ and $C_5 \leq V_p (p) \leq C_6$ for $p
\in [0, 1]$.

\noindent The positive constants $C_1, \ldots, C_6$ depend only on $\omega_0$ 
(see the first formula \eqref{omegaL} for its definition).
\end{proposition}

\begin{proof}
It is easy to see that the right-hand side terms in \eqref{KK10}, \eqref{KK9} and
\eqref{KK7}--\eqref{KK8a} depend monotonically on both $\frak m$ and $\frak M$.
Therefore, these estimates imply that the inequalities listed in items (i)--(iii)
follow from the inequalities
\begin{equation} \label{KK7c}
C' (\omega_0) \leq \lambda_c \leq C'' (\omega_0) 
\end{equation}
provided the estimates
\begin{equation} \label{KK7b}
\frak m \geq \frak m_* \quad \mbox{and} \quad \frak M \leq \frak M_* 
\end{equation}
hold with positive constants $\frak m_*$ and $\frak M_*$ depending only on $\omega_0$. 

The definition of $\lambda_c$ implies that $\lambda_c^2 - \lambda_0^2 \leq 1$, and
so we have that
\[ \frak m \geq \frac{1}{\sqrt{\lambda_c^2 - \lambda_0^2 + \lambda_0^2 - 2 \min_{\tau 
\in [0,1]} \Omega (\tau)}} \geq \frac{1}{\sqrt{1 + 4 \, \omega_0}} ,
\]
which gives the first inequality \eqref{KK7b}. Furthermore, we see that
\[ \frak M \leq \frac{2}{\sqrt{\lambda_c^2 - \lambda_0^2}} .
\]
Hence it is sufficient to estimate $\lambda_c^2 - \lambda_0^2$ from below in order
to obtain the second inequality \eqref{KK7b}. Let $\tau_0$ be such that $2 \Omega 
(\tau_0) = \lambda_0^2$, then we have
\[ \int_0^1 (\lambda_c^2 - \lambda_0^2 + 2 \omega_0 |\tau - \tau_0|)^{-3/2} \D \tau
\leq \int_0^1 [ \lambda_c^2 -2 \Omega (\tau) ]^{-3/2} \D \tau = 1 .
\]
Evaluating the integral on the left-hand side, we get the inequality
\[ \frac{1}{\omega_0} \Big[ 2 (\lambda_c^2 - \lambda_0^2)^{-1/2} -
(\lambda_c^2 - \lambda_0^2 + 2 \omega_0 (1 - \tau_0))^{-1/2}
- (\lambda_c^2 - \lambda_0^2 + 2 \omega_0 \tau_0)^{-1/2} \Big] \leq 1 .
\]
It implies that $(\lambda_c^2 - \lambda_0^2)^{-1/2} - (\lambda_c^2 - \lambda_0^2 +
\omega_0 )^{-1/2} \leq \omega_0$, thus yielding
\[ 1 \leq 2 (\lambda_c^2 - \lambda_0^2 + \omega_0) \sqrt{\lambda_c^2 - \lambda_0^2} .
\]
Hence either $1/4 \leq (\lambda_c^2 - \lambda_0^2)^{3/2}$ or $1/4 \leq \omega_0
(\lambda_c^2 - \lambda_0^2)^{1/2}$, and so
\[ \frac{1}{4^{2/3} + (4 \omega_0)^2} \leq \lambda_c^2 - \lambda_0^2 ,
\]
which gives the second inequality \eqref{KK7b}.

Furthermore, combining the last inequality and $\lambda_c^2 - \lambda_0^2 \leq 1$,
we obtain \eqref{KK7c} with
\[ C' = \sqrt{\lambda_0^2 + \left[ 4^{2/3} + (4 \omega_0)^2 \right]^{-1}} \quad
\mbox{and} \quad C'' = \sqrt{1 + \lambda_0^2} \, .
\]
This completes the proof of the proposition.
\end{proof}

\section{Properties of Solutions of Problem ${\rm P_r^M}$}

In this section, we study the following properties of solutions. First, we obtain a
lower bound for $\check \eta = \inf_{x \in \Bbb R} \eta (x)$, which is similar to
the first assertion of Theorem 1, \cite{KKL2}, where the inequality $\check{\eta}
\geq d_- (r)$ was proved. (We cannot use that theorem here because it was
established under stronger assumptions than those imposed in problem ${\rm P_r^M}$.)
Then we derive a uniform bound for $\nabla \psi$ that depends only on $\omega_0$
provided Bernoulli's constant is separated from $r_0$. Moreover, bounds are obtained
for the H\"older norms of both $\psi$ and $\eta$. Finally, it is shown that the
smallness of $\eta - d_+ (r)$ depends on the difference $r - r_c$.

In what follows, we use different local estimates for solutions of  problem ${\rm
P_r^M}$ near the free surface, for which purpose we have to ascertain how thick is
the domain $D$.

\begin{lemma} \label{EtaBound}
Let $(\psi, \eta)$ be a solution of problem ${\rm P_r^M}$, then  the following
inequality holds: $\check{\eta} \geq \min \left\{ (6 r)^{-1/2} , (2 \omega_0)^{-1/2}
\right\}$.
\end{lemma}

\begin{proof} 
Let $d > 0$ and $\epsilon \in (0, 1)$ be given numbers, and let $\eta_{\epsilon} (x)
= \epsilon \, \E^{-x^2}$. In the domain $D_{d,\epsilon} = \{ (x,y) \in \Bbb R^2 : 0
< y < d + \eta_{\epsilon} (x) \}$, we consider the auxiliary boundary value problem
\[ \nabla^2 \Psi + \omega_0 = 0 \ \text{in} \  D_{d,\epsilon} , \quad
\Psi (x,0) = 0, \ \ \ \Psi (x, d + \eta_{\epsilon} (x)) = 1 ,
\]
whose bounded solution we denote by $\psi_{d,\epsilon}$. Let us show that
\begin{equation}
\| \psi_{d,\epsilon} - u_{d,\epsilon} \|_{C^1 (D_{d,\epsilon})} \leq \epsilon \, 
C_{d, \omega_0} ,
\label{3sep}
\end{equation}
where $u_{d,\epsilon} (x, y) = u_d ( d y / [d + \eta_{\epsilon} (x)] )$ and $u_d (y)
= y [ d^{-1} + \omega_0 (d - y) / 2 ]$. It is clear that the last function satisfies
the following relations:
\[ u_d'' + \omega_0 = 0 \ \mbox{on} \ (0, 1) , \quad u_d (0) = 0 , \quad u_d (d) = 1 .
\]
Moreover, it is an increasing function on $[0, d]$ when $d < d_* = \sqrt{2 /
\omega_0}$. In order to prove \eqref{3sep}, we note that $u = \psi_{d,\epsilon} -
u_{d,\epsilon}$ solves the homogeneous Dirichlet problem for the following equation:
\[ \nabla^2 u = \epsilon \left\{ \left[ \frac{2 d x \E^{-x^2}} {d + \epsilon 
\E^{-x^2}} \, u'_d \left( \frac{d y}{d + \eta_{\epsilon} (x)} \right) \right]_x
\!\! - \frac{\omega_0 \, \E^{-x^2} [2 d + \eta_{\epsilon} (x)]}{[d + \eta_{\epsilon}
(x)]^2} \right\} \ \ \mbox{in} \ D_{d,\epsilon} .
\]
This yields that $\| u \|_{W^{2,1} (D_{d,\epsilon})} \leq \epsilon \, C_{d, \omega_0}$,
from which \eqref{3sep} follows by virtue of local estimates.

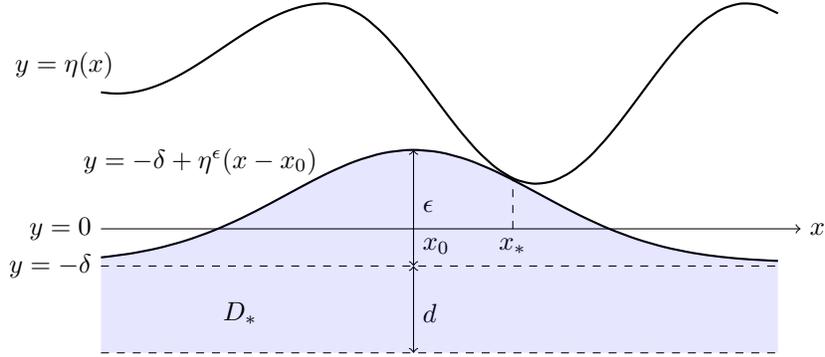
\begin{figure} \label{Pic2}
\centering
\begin{tikzpicture}[scale=1.50]
\draw[->] (-2,1.8) node[left]{$y=0$} -- (4.2,1.8) node[right]{$x$};
\draw[-,dashed] (-2,1.47) node[left]{$y=-\delta$} -- (4,1.47);
\draw[-,dashed] (-2,0.7) -- (4,0.7);

\draw[-,dashed] (1.65,1.8) node[below]{$x_*$} -- (1.65,2.26);

\draw[<->] (0.77,0.7) -- (0.77,1.47);
\draw[<->] (0.77,1.47) -- (0.77,2.5) ;

\node[below right] at (0.77,1.8) {$x_0$};

\node[right] at (0.77,2) {$\epsilon$};
\node[right] at (0.77,1.05) {$d$};

\node[left] at (-1.80,3.25) {$y=\eta(x)$};
\node[left] at (0,2.4) {$y=-\delta + \eta^{\epsilon}(x-x_0)$};

\draw[thick] plot[domain=-2:4,smooth] (\x,{1.5+exp(-0.4*(\x-0.77)*(\x-0.77))});

\draw[thick] plot[domain=0:4,smooth] (\x,{3+0.8*cos(1.69*\x r)});
\draw[thick] plot[domain=-2:0,smooth] (\x,{3.4+0.4*cos(1.69*\x r)});

\fill[fill=blue, opacity = 0.1] (-2,0.7) -- plot[domain=-2:4,smooth]
(\x,{1.5+exp(-0.4*(\x-0.77)*(\x-0.77))}) -- (4,0.7) -- cycle;

\node[right] at (-1,1.05) {$D_*$};

\end{tikzpicture}
\caption{A sketch of the domain $D_*$.}
\end{figure}

Now we turn to the lower bound for $\check{\eta}$ and consider two cases. First,
let us assume that $\check{\eta} = 0$. Since $\eta$ does not vanish identically,
for some $d \in (0, d_*)$ and $\delta > 0$ there exists $\epsilon \in (\delta, 2
\delta)$ and $x_0 , x_* \in \Bbb R$ (see Figure 5) such that
\[ \eta (x) \geq - \delta + \eta_{\epsilon} (x - x_0) \quad \mbox{and} \quad
\eta (x_*) = - \delta + \eta_{\epsilon} (x_* - x_0) .
\]
Let us apply the maximum principle to the superharmonic function 
\[ U (x,y) = \psi_{d, \epsilon} (x-x_0, y+d+\delta) - \psi (x, y)
\]
in the domain 
\[ D_* = \{ (x,y) : - Q < x - x_0 < Q , \ \ 0 < y < - \delta + \eta_{\epsilon} 
(x - x_0) \} ,
\]
where $Q > 0$ is such that $\eta_{\epsilon} (Q) = \delta$. Since $U$ is positive
in this domain and $U (x_*, - \delta + \eta_{\epsilon} (x_* - x_0)$ vanishes, we
obtain that $\partial_n U (x_*, - \delta + \eta_{\epsilon} (x_* - x_0)) < 0$,
and so $u_d' (d) \leq \sqrt{3 r} + O (\delta)$. Letting $\delta \to 0$ in this
inequality, we arrive at $u_d' (d) \leq \sqrt{3 r}$, which is impossible when
$d$ is sufficiently small. The obtained contradiction shows that $\check{\eta} >
0$.

Let us assume that the positive $\check{\eta}$ is less than $(2 \omega_0)^{-1/2}$, because otherwise the required inequality is obviously true. To keep the same
notation as in the previous case we put $d = \check{\eta}$. Then for every $\delta
\in (0, d)$ there exists $\epsilon \in (\delta, 2 \delta)$ and $x_0 , x_* \in \Bbb
R$ such that
\[ \eta (x) \geq d - \delta + \eta_{\epsilon} (x - x_0) \quad \mbox{and} \quad
\eta (x_*) = d - \delta + \eta_{\epsilon} (x_* - x_0) .
\]
Now we apply the maximum principle to the superharmonic function 
\[ U (x,y) = \psi_{d, \epsilon} (x-x_0, y+\delta) - \psi (x, y)
\]
in the domain $\{ (x,y) : -\infty < x < +\infty , \ \ \ 0 < y < d - \delta +
\eta_{\epsilon} (x - x_0) \}$, thus concluding (similarly to the previous case)
that $\partial_n U (x_*, d - \delta + \eta_{\epsilon} (x_* - x_0)) < 0$. Therefore,
\[ u_d' (d) =  \check{\eta}^{-1} - \omega_0 \check{\eta} / 2 \leq 
\sqrt{3 r - \check{\eta}} ,
\]
and so $\check{\eta} \geq \min \left\{ (6 r)^{-1/2} , (2 \omega_0)^{-1/2} \right\}$.
It should be noted that this inequality for $\check{\eta}$ also holds when
$\check{\eta} \geq d_0$, which completes the proof.
\end{proof}

\subsection{Uniform Bound for Velocity Field} 

Here we give a uniform bound for $\nabla \psi$, and the main difficulty is the
nonlinear term in equation \eqref{eq:lapp} which does not allow us to apply the
standard maximum principle for elliptic equations.

\begin{proposition} \label{PropBoundsPsi} 
For every $R > r_c$ there exists a constant $C (\omega_0, R)$ (it does not depend on
other parameters) such that the inequality $|\nabla \psi (x, y)| \leq C (\omega_0,
R)$ holds for all $(x, y) \in \bar D$ provided $(\psi, \eta)$ solves problem ${\rm
P_r^M}$ with $r \in (r_c, R]$.
\end{proposition}

\begin{proof}
Let $D_t = \{ (x,y): t \leq x \leq t+1, 0 \leq y \leq \eta(x) \}$ for an
arbitrary $t \in \Bbb R$. Our first aim is to show that 
\begin{equation} \label{L2bound}
\sup_{t \in \Bbb R} \int_{D_t} |\nabla \psi|^2 \D x \D y \leq C (\omega_0, M, r) ,
\end{equation}
where the constant $C (\omega_0, M ,r)$ does not depend on other parameters. For this
purpose we multiply the equation \eqref{eq:lapp} by $\psi$, integrate the result
over $D_t$ and apply the first Green's formula, thus obtaining
\begin{eqnarray}
&& \int_{D_t} \!\! |\nabla \psi|^2 \D x \D y = \int_{D_t} \!\! \omega (\psi) \psi \,
\D x \D y + \int_t^{t+1} \!\! \nabla \psi \cdot (-\eta' (x), 1) \, \D x \nonumber \\
&& + \int_0^{\eta (t+1)} \!\! \psi_x (t+1,y) \psi (t+1,y) \, \D y - \int_0^{\eta
(t)} \!\! \psi_x (t,y) \psi (t,y) \, \D y .
\label{l2bound1} 
\end{eqnarray}
Let us consider each of the four integrals on the right-hand side, say
$I_1,\,I_2,\,I_3,\,I_4$.

Since the values of $\psi$ belong to $[0, 1]$ according to \eqref{eq:uni}, the
definition of $\omega_0$ implies that $|I_1| \leq 3 r \omega_0 / 2$ for all $t \in
\Bbb R$ because $\eta (x) \leq 3 r / 2$ for all $x \in \Bbb R$. Furthermore, the
boundary condition \eqref{eq:kcp} yields that
\[ \psi_x (x, \eta (x)) = - \eta' (x) \psi_y (x, \eta (x)) ,
\]
and so the Bernoulli equation \eqref{eq:bep} gives 
\[ 0 < I_2 = \int_t^{t+1} \sqrt{[3r - 2 \eta (x)] [1 + \eta'^2 (x)]} \, \D x \leq 
\sqrt{3 r (1 + M^2)} \quad \mbox{for all} \ t \in \Bbb R .
\]

Our estimates of $|I_3|$ and $|I_4|$ are based on the equality
\[ \int_0^{\eta (t)} \!\! \psi_x(t,y) \psi(t,y) \, \D y = 2^{-1} \frac{\D}{\D t}
\left[ \int_0^{\eta (t)} \!\! \psi^2 (t,y) \, \D y  -  \eta (t) \right] ,
\]
which allows us to estimate
\[ \left| \int_{\Bbb R} \E^{-|t-\tilde{t}|} \, \D \tilde{t} 
\int_0^{\eta (\tilde{t})} \psi_x (\tilde{t},y) \psi (\tilde{t},y) \, \D y \,
\right| \quad \mbox{for all} \ t \in \Bbb R .
\]
Indeed, multiplying the right-hand side of the previous formula by the
corresponding exponential and integrating over $\Bbb R$, we get after integration by
parts that the last integral is less than or equal to $3r$.

Thus, it follows from \eqref{l2bound1} and the obtained inequalities that
\[ \int_{\Bbb R} \E^{-|t-\tilde{t}|} \, \D \tilde{t} \int_{D_{\tilde{t}}} 
|\nabla \psi|^2 \, \D x \D y \leq C (\omega_0, M, r) \quad \mbox{for all} \ t
\in \Bbb R .
\]
After changing the order of integration, this gives
\[
\int_{D} |\nabla \psi|^2  \E^{-|x-t|} \, \D x \D y \leq C (\omega_0, M, r) \quad
\mbox{for all} \ t \in \Bbb R ,
\]
which, in its turn, yields \eqref{L2bound}.

To complete the proof we apply some local estimates. First, we concentrate on
interior estimates and estimates near the bottom for which purpose we use only
the equation \eqref{eq:lapp} and the boundary condition \eqref{eq:bcp}.

The Schauder interior estimates (see \cite{GT}, Theorem 6.2) and the inequality
\eqref{L2bound} imply that $|\nabla \psi|$ is bounded pointwise by a constant $C
(\omega_1, \epsilon)$ in the domain $D_\epsilon$ whose points are distant from
$\partial D$ not less than $\epsilon  = \check{\eta}/3$. Besides, combining local
estimates near a smooth boundary (see \cite{GT}, Corollary 8.36) and the inequality
\eqref{L2bound}, one obtains that $|\nabla \psi|$ is bounded pointwise by a constant
$C(\omega_0, \epsilon)$ in the $\epsilon$-neighbourhood of the bottom $\{ x \in \Bbb
R, y = 0 \}$. Furthermore, Lemma \ref{EtaBound} shows that $\epsilon =
\check{\eta}/3$ is greater then a positive constant depending only on $\omega_0$ and
$R$.

To estimate $\nabla \psi$ near the free surface we use Theorem 8.25, \cite{GT}, for
solutions of the equation $\nabla^2 \psi^* + \omega' (\psi) \psi^* = 0$, thus
obtaining
\begin{equation} \label{GTTh25}
\sup_{B_a (Z_0)} \psi^*_m \leq C (\omega_1) a^{-1} \| \psi^*_m \|_{L^2 (B_{2 a} (Z_0))} .
\end{equation}
Here $\psi^*$ stands for either of the first derivatives of $\psi$, $a > 0$ is fixed
and $Z_0$ is an arbitrary point of $D$, $B_a (Z_0)$ denotes the open circle of
radius $a$ centred at $Z_0$, $m = \sup_{\partial D \cap B_{2 a} (Z_0)} \psi^*$ and
\[ \psi^*_m (x,y) =  \begin{cases} \max \{ \psi^* (x,y), m \} \ \mbox{when} \ (x,y) 
\in D , \\ m \ \mbox{when} \ (x,y) \notin D . \end{cases}
\]
Applying \eqref{GTTh25} in an arbitrary ball centred at the free surface with $a =
\epsilon = \check{\eta}/3$, we note that the Bernoulli equation \eqref{eq:bep}
implies that $m \leq \sqrt{3 r}$. Combining \eqref{L2bound} and \eqref{GTTh25}, we
find that $\psi^*$ is bounded from above by $C (\omega_1, M, r)$ in the
$\epsilon$-neighbourhood of the free surface. Since the same argument is valid for
$- \psi^*$, we get that
\[ |\nabla \psi (x, y)| \leq C (\omega_1, M, r) \quad \mbox{for all} \ (x, y)
\in D .
\]

In order to prove that a similar estimate holds with a constant independent of $M$
we consider the function
\[ P(x,y) =  - 3r + |\nabla \psi|^2 + 2 y - 2 \Omega (\psi) .
\]
A direct calculation gives that $\nabla^2 P \geq \omega_0^2$ in $D$, and so $P^* = P
+ \omega_0^2 y^2/2$ is a subharmonic function in $D$. Since $P^*$ is bounded, the
maximum principle (see \cite{C}, Theorem 1.4) is applicable. Hence the supremum of
$P^*$ is attained on the boundary. On the free surface, $P^*$ is bounded by $2
\omega_0 + 3 \omega_0^2 r^2$, whereas there exists a constant $C(\omega_0, R)$ such
that $P^* \leq  C(\omega_0, R)$ on the bottom. Therefore, $P^* \leq C(\omega_0, R)$
throughout $D$, which proves the proposition when one takes into account the
definition of $P$.
\end{proof}

\subsection{Bounds for Solutions of $\rm P_r^M$ in the Nearcritical Case}

By the definition of problem $\rm P_r^M$, the slope of $\eta$ is bounded, but this
restriction does not prevent that stagnation points might be present on $\eta$.
However, if the Bernoulli constant $r$ is sufficiently close to its critical value
$r_c$, then any solution of ${\rm P_r^M}$ is of small amplitude, and so there are
no stagnation points on $\eta$. The following assertion deals with both these
properties.

\begin{theorem} \label{propSmall} 
Let $M$ be positive and $r$ is subject to condition \eqref{BdR}. Then the following assertions are true.
\begin{itemize}
\item[\rm (a)] For every $\epsilon > 0$ there exists a constant $r' (\epsilon, M,
\omega_0) > r_c$ such that the inequality $\sup_{x \in \Bbb R} |\eta(x) - d_+(r)| <
\epsilon$ is valid for the second component of a solution of problem\/ ${\rm P_r^M}$
with $r \in (r_c, r']$. Here $d_+(r)$ is the quantity defined in Section 2.1.
\item[\rm (b)]  There exist $r'' (M, \omega_1) > r_c$ and $\delta (M, \omega_1) > 0$
such that the inequality $\psi_y > \delta(M, \omega_1)$ is fulfilled in $D$ for the
first component of a solution of problem\/ ${\rm P_r^M}$ with $r \in (r_c, r'']$.
\end{itemize}
\end{theorem}

\begin{proof} 
To prove (a), it is reasonable to use $w (q,p) = h (q,p) - H (p)$ because $\eta (x)
= h (x, 1)$, $x \in \Bbb R$, and $d_+ (r) = H (1, \lambda_+ (r))$; here $h$ is a
solution of the problem \eqref{eq:hdf1}--\eqref{eq:hdf3} in $S = \{ (q,p): q \in
\Bbb R; p \in (0,1) \}$ and $H$ is defined by formula \eqref{Stream}. Relations
\eqref{PsiHrel} between the derivatives of $\psi$ and $h$ yield that $h$,
$h_p^{-1}$ and $f = h_q h_p^{-1}$ are bounded by a constant depending on $\omega_0$
and $R$ (see Proposition \ref{PropBoundsPsi}).

Let $\phi_0$ be the eigenfunction of the problem \eqref{SL1}--\eqref{SL3}
corresponding to $\mu_0$. If $H = H(p; \lambda_+(r))$, then $\phi_0$ is negative
(see assertion (i) of Proposition \ref{PropK1}). Putting
\[ w_0 (q) = \int_0^1 w (q,p) \phi_0 (p) H_p^{-1} (p) \, \D p \quad \mbox{and} 
\quad f_0 (q) = \int_0^1 f (q, p) \phi_0 (p) \, \D p ,
\]
we get from \eqref{eq:hdf1}--\eqref{eq:hdf3}
\[ \int_0^1 \left[ \frac{H_p^3 f^2}{2} + \frac{(2h_p+H_p) w_p^2}{2h_p^2}
\right] \frac{\phi_0'}{H_p^3} \, \D p = \mu_0 w_0 - f_0' .
\]
Multiplying this identity by $\E^{-\vartheta |q-q_0|}$ with some $\vartheta > 0$ and
$q_0 \in \Bbb R$, and integrating the result over $\Bbb R$, we estimate both
integrals on the right-hand side, thus obtaining
\begin{eqnarray}
&& \int_{S} \left[ \frac{H_p^3 f^2}{2} + \frac{(2h_p+H_p) w_p^2}{2h_p^2} \right]
\frac{ \phi_0'}{H_p^3} \E^{-\vartheta |q-q_0|} \, \D q \D p \nonumber \\ && \leq
|\mu_0| \int_{-\infty}^{+\infty} |w_0| \, \E^{-\vartheta |q-q_0|} \, \D q + C_4
(\omega_0) \vartheta \int_S |f| \, \E^{-\vartheta |q-q_0|} \D q . \label{small:1}
\end{eqnarray}
Here $C_4$ is the same constant as in Proposition \ref{PropK1}, and it arises after
integration by parts in the second integral. Since $w_0$ and $f$ are bounded on
$\bar S$, both integrals on the right-hand side are convergent. Therefore, the first
term is bounded by $C(\omega_0) |\mu_0| / \vartheta$, whereas the second one is less
than or equal to
\[ C (\omega_0) \sqrt{\vartheta} \left[ \int_{S} f^2 \E^{-\vartheta |q-q_0|} \, 
\D q \D p \right]^{1/2} ,
\]
which is a consequence of the Schwarz inequality. Putting $\vartheta =
\sqrt{|\mu_0|}$, we get from \eqref{small:1}
\[ I^2 - C (\omega_0) |\mu_0|^{1/4} I - C (\omega_0) \sqrt{|\mu_0|} \leq 0 , 
\quad \mbox{where} \ I = \left[ \int_{S} f^2 \E^{-\theta |q-q_0|} \right]^{1/2} .
\]
Here, it is also taken into account that $\min_{p \in [0,1]} \phi' (p) \geq C$ (see
assertion (iii) of Proposition \ref{PropK1}); here the last constant is positive and
depends only on $\omega_0$. Then the inequality for $I$ gives that $I \leq C
(\omega_0) |\mu_0|^{1/4}$. Combining this and \eqref{small:1}, we arrive at the
following inequality:
\begin{equation} \label{small:2}
\int_{S} \left[ \frac{H_p^3 f^2} {2} + \frac{(2h_p+H_p) w_p^2}{2h_p^2} \right] 
\frac{ \phi_0'}{H_p^3} \E^{-\theta |q-q_0|} \, \D q \D p \leq C (\omega_0) 
\sqrt{|\mu_0|} .
\end{equation}

Furthermore, by virtue of the Schwarz inequality, we get that
\[ \int_{q_0-1}^{q_0+1} \!\! \int_0^1 |w_p| \, \D p \D q \leq \left[ 
\int_{q_0-1}^{q_0+1} \!\! \int_0^1 \frac{|w_p|^2}{h_p} \, \D p \D q \right]^{1/2}
\left[ \int_{q_0-1}^{q_0+1} \!\! \int_0^1 h_p \, \D p \D q \right]^{1/2} ,
\]
and the last expression is bounded by $C(\omega_0) |\mu_0|^{1/4}$ in view of
\eqref{small:2}. This fact and the inequality $|w(q,1)| \leq \int_0^1 |w_p (q, p)|
\, \D p$ give us that
\[ \int_{q_0-1}^{q_0+1} |\eta(q) - d_+(r)| \, \D q \leq C (\omega_0) |\mu_0|^{1/4} 
\quad \mbox{for all} \ q_0 \in \Bbb R ,
\]
because $w(q,1) = \eta(q) - d_+(r)$.

According to assertion (i) of Proposition \ref{PropK1}, we have that $\mu_0
(\lambda) \to 0$ as $\lambda \to \lambda_c$, and so the last inequality implies that
for every $\epsilon > 0$ there exists $r' = r' (\epsilon, M, \omega_0)$ such that
$|\eta (q) - d_+ (r)| < \epsilon$ for all $q \in \Bbb R$ and all $r \in (r_c, r')$.
Indeed, $|\eta' (q)| \leq M$ a.e. on $\Bbb R$ being the second component of a
solution of problem ${\rm P_r^M}$.

\vspace{2mm}

Now we turn to assertion (b) for a solution of problem ${\rm P_r^M}$. First, we show
that $\psi_y$ is separated from zero on $\partial D$ when $r$ is close to $r_c$.

Since the boundary conditions \eqref{eq:kcp} and \eqref{eq:bep} imply that
\[ \left\{ 1 + [\eta'(x)]^2 \right\} \psi_y^2 (x, \eta(x)) = 3r - 2\eta (x) 
\quad \mbox{for all} \ x \in \Bbb R ,
\]
we have
\[ [\psi_y (x, \eta(x))]^2 \geq [3r - 2\eta (x)] / (1+M^2) \quad \mbox{for all} 
\ x \in \Bbb R .
\]
Let us show that the expression on the right-hand side is separated from zero by a
positive constant depending only on $\omega_0$. For this purpose we write
\[ 3r - 2 \eta = 3r - 2d_+(r) + 2 [\eta - d_+(r)] = [\lambda_+ (r)]^2 - 2 \Omega (1) 
+ 2 [\eta - d_+(r)] ,
\]
where the last equality is a consequence of \eqref{eq:calR}. In view of
\eqref{BdR} we have
\[ [\lambda_+ (r)]^2 - \Omega (1) = \lambda_0^2 - \Omega(1) + [\lambda_+ (r)]^2 
- \lambda_0^2 \geq (\lambda_c - \lambda_0) \lambda_c / 4 ,
\]
but assertion (a) of Theorem \ref{propSmall} with $\epsilon = (\lambda_c -
\lambda_0) \lambda_c / 16$ yields that there exists $r'(M,\omega_0)$ such that
\[ 2 |\eta - d_+(r)| \leq (\lambda_c - \lambda_0) \lambda_c
/ 8 \quad \mbox{for all} \ r \in (r_c,r'] .
\]
Therefore,
\[ [\psi_y (x, \eta(x))]^2 \geq (\lambda_c - \lambda_0) \lambda_c / [8 (1+M^2)]
\quad \mbox{for all} \ x \in \Bbb R .
\]

To estimate $\psi_y$ near the bottom we use the problem
\eqref{eq:hdf1}--\eqref{eq:hdf3} because the maximum principle (see \cite{C},
Theorem 1.4) is applicable to it. We choose $\lambda_*$ so that $\sup_{x \in \Bbb R}
\eta (x) = H(p; \lambda_*)$, and put $w (q, p) = h (q, p) - H (p; \lambda_*)$. Since
$w (q, 0) = 0$ for all $q \in \Bbb R$ and $w (q, 1) < 0$, the maximum principle
implies that $w < 0$ in $S$ and $w_p (q, 0) < 0$. In other words, $h_p (q,0) < H_p
(0, \lambda_*)$, and so
\[ \psi_y(x,0) \geq [H_p (0, \lambda_*)]^{-1} \geq [H_p (0,\lambda_+(r))]^{-1} \geq
\sqrt{(\lambda_c - \lambda_0) \lambda_c / 4} .
\]
Here the second inequality is a consequence of \eqref{PsiHrel}.

The obtained estimates of $\psi_y$ on $\partial D$ allow us to apply Theorem 8.26,
\cite{GT}, to the equation $\nabla^2 \psi_y + \omega'(\psi) \psi_y = 0$, thus
demonstrating that $\psi_y$ is separated from zero in a neighbourhood of $\partial
D$. According to this theorem we have
\begin{equation} \label{T4:3}
\| \psi_y^- \|_{L^2 (B_{6 a} (Z))} \leq C (\omega_1) \inf_{B_{3 a} (Z)} \psi_y^- ,
\end{equation}
where $Z$ is an arbitrary point of $\partial D$,
\[ \psi_y^{-} (x, y) = \begin{cases} \min \{\psi_y, m\} \ \mbox{when} \ (x,y) \in D, 
\\ m \ \mbox{when} \  (x,y) \notin D \end{cases}
\]
and $m = \inf_{B_{12 a} (Z) \cap \partial D} \psi_y$. We recall that $B_a (Z)$
denotes the open circle of radius $a$ centred at $Z$.

It is essential that the constant in \eqref{T4:3} depends only on $\omega_1$ and is
independent of $a$. Indeed, Lemma \ref{EtaBound} and inequality \eqref{BdR} imply
that 
\[ C_1 (\omega_0) < 3 a < C_2 (\omega_0) ,
\]
where both constants are positive and depend only on $\omega_0$. Since the free
surface is a Lipschitz curve with the constant $\sqrt{1+M^2}$, a large part of the
ball $B_{6 a} (Z)$ belongs to the complement of $D$ for every $Z \in \partial D$
(this is obvious for the flat bottom), and so $\| \psi_y^- \|_{L^2 (B_{6 a} (Z))}
\geq C (M) \, m$. Combining this inequality and \eqref{T4:3}, we get that
\[ \inf_{B_{3 a} (Z) \cap D} \psi_y \geq C (M, \omega_1) m \quad \mbox{for every} \ 
Z \in \partial D .
\]
The last expression is greater than a certain positive constant $C_1 (M, \omega_1)$,
and so the required inequality holds in the whole domain $D$ because it is covered
by circles $B_{3 a} (Z)$ with $Z \in \partial D$. This completes the proof.
\end{proof}

Now, we are going to estimate H\"older norms of derivatives for solutions of problem
${\rm P_r^M}$. A bound for them will be obtained under the assumption that the
horizontal component of the velocity field is uniformly separated from zero in $\bar
D$.

\begin{proposition} \label{RegProp} 
Let $R \in (r_c, r_0)$ and $M > 0$. If problem\/ ${\rm P_r^M}$ with $r \in (r_c, R]$
has a solution $(\psi, \eta)$, whose first component satisfies the inequality
$\psi_y > \delta$ in $\bar D$ with some $\delta > 0$, then there exist $\alpha \in
(0,1)$ and $C > 0$ (both depending only on $M$, $R$, $\delta$ and $\omega_1$) such
that the following inequalities hold:
\begin{equation} \label{eq:glob}
\| \psi \|_{C^{2,\alpha} (\bar D)}  \leq C \quad \mbox{and} \quad \| \eta \|_{C^{2,\alpha} 
(\Bbb R)} \leq C .
\end{equation}
\end{proposition}

\begin{proof} 
The assumptions imposed on $(\psi, \eta)$ and Proposition \ref{PropBoundsPsi} imply
that the corresponding solution $h$ of problem \eqref{eq:hdf1}--\eqref{eq:hdf3}
satisfies the following inequalities
\begin{equation} \label{restr}
0 < C (\omega_1, R) \leq h_p \leq \delta^{-1} \ \ \ \text{in} \ \ \ \bar S = \Bbb R
\times [0,1] .
\end{equation}
Let us show that there exists $\alpha > 0$ such that the inequality
\begin{equation} \label{C1alpha}
\| h \|_{C^{1,\alpha} (\bar S)} \leq C' (R, \omega_1) 
\end{equation}
holds with a positive constant $C' (R, \omega_1)$.

Since $\partial S$ consists of two straight lines and the Dirichlet boundary
condition is fulfilled when $p = 0$, the only difficulty is to prove the estimate
near the line $\{ q \in \Bbb R , p = 1 \}$ (cf. \cite{GT}, Theorems 13.1 and 13.2).
To overcome this difficulty we use the local estimates obtained in \cite{LU} (see
the proof of Theorem 2.1, ch. 10) for the quasilinear equation \eqref{eq:hdf1}
written in divergence form and complemented by the boundary condition
\eqref{eq:hdf3}; namely:
\begin{eqnarray}
&& \!\!\!\!\!\!\!\!\!\!\!\!\!\!\!\!\!\!\!\!\!\! [a_1 (h_q, h_p)]_q + [a_2 (h_q,
h_p)]_p - \omega = 0 , \quad [a_2 (h_q, h_p) + \phi (h)]_{p=1} = 0 ,
\label{16sept1} \\ && \!\!\!\!\!\!\!\!\!\!\!\!\!\!\!\!\!\!\!\!\!\! \mbox{where}
\ a_ 1 (h_q, h_p) = \frac{h_q}{h_p} , \ a_2 (h_q, h_p) = - \frac{1 + h_q^2}{2
h_p^2} \ \mbox{and} \ \phi (h) = 2h - 3r . \label{16sept2}
\end{eqnarray}
Indeed, all conditions on $a_1$, $a_2$ and $\phi$ required in the mentioned theorem
are fulfilled in our case because the imposed conditions guarantee that the
two-sided restriction \eqref{restr} holds which, in its turn, yields
\eqref{C1alpha}.

The next step is to apply Theorem 11.2, \cite{ADN}, to the problem \eqref{16sept1}
and \eqref{16sept2}, which gives that $h \in C^{2,\alpha} (\bar S)$, but does not
provide a bound for the norm. To obtain such a bound we write the problem as
follows:
\begin{eqnarray*}
&& h_{qq} - 2h_q h_p^{-1} h_{qp} + (1+h_q^2) h_p^{-2} h_{pp} + \omega (p) h_p = 0 \
\ \ \text{in} \ S ; \\ && (1+h_q^2)h_p^{-2} - (3r - 2h) = 0 \ \ \text{when} \ p = 1;
\quad h = 0 \ \ \text{when} \ p = 0 .
\end{eqnarray*}
Now, we are in a position to use Theorem 3.1, \cite{LU}, ch. 10, which gives a bound
for $\| h \|_{C^{2,\alpha} (\bar S)}$, depending only on $M$, $R$, $\delta$ and
$\omega_1$. Therefore, in view of \eqref{PsiHrel} and \eqref{restr} both
inequalities \eqref{eq:glob} are true.
\end{proof}

\section{Proof of Theorem \ref{StabThm}}

Let us prove the following assertion that is slightly stronger than Theorem
\ref{StabThm}.

\begin{theorem} \label{StabThm2'}
There exist $r'' \in (r_c, r_0)$ and positive $C$ and $\theta$ $(r''$ depends on $M$
and $\omega_1$, whereas $C$ and $\theta$ depend only on $\omega_0)$ such that any
two solutions $(\psi^{(1)}, \eta^{(1)})$ and $(\psi^{(2)}, \eta^{(2)})$ of problem
$\rm P_r^M$ with $r \in (r_c, r'']$ satisfy the inequality
\begin{eqnarray} 
&& \!\!\!\!\!\!\!\!\!\!\!\!\!\!\!\!\!\!\!\!\!\!\!\! \int_{S}  \big| \nabla h^{(1)} -
\nabla h^{(2)} \big|^2 \, \E^{-\theta |x-x_0|} \, \D q \D p \nonumber \\ && \leq C
\left[ |\eta^{(1)} (q_0) - \eta^{(2)} (q_0)|^2 + |\eta_x^{(1)} (q_0) - \eta_x^{(2)}
(q_0)|^2 \right] . \label{eq:stab'}
\end{eqnarray}
Here $q_0 \in \Bbb R$ is arbitrary, whereas $h^{(1)}$ and $h^{(2)}$ correspond to
$(\psi^{(1)}, \eta^{(1)})$ and $(\psi^{(2)}, \eta^{(2)})$ through the partial
hodograph transform.
\end{theorem}

It is easy to see that Theorem \ref{StabThm} is a consequence of Theorem
\ref{StabThm2'}. Our proof of the latter theorem is based on another form of problem
\eqref{eq:hdf1}--\eqref{eq:hdf3}; namely, a system of Hamilton's equations proposed
in \cite{GW} (see \eqref{eq:v1} and \eqref{eq:v2} below).

\vspace{-2mm}

\subsection{A System of First Order \\ Equivalent to Problem
\eqref{eq:hdf1}--\eqref{eq:hdf3}}

Let us write a system equivalent to the problem \eqref{eq:hdf1}--\eqref{eq:hdf3},
for which purpose we complement $h$ by the unknown function $f = h_q / h_p$, thus
obtaining for $(q, p) \in S$:
\begin{numcases}{}
h_q = f h_p , \label{eq:v1} \\
f_q = - \frac 1 2 {\cal A} (f, h) + \omega (p) . \label{eq:v2}
\end{numcases}
Here $\cal A$ is a nonlinear operator which is convenient to define by virtue of the
integral identity
\[ \int_0^1 {\cal A}( f, h) \Phi \, \D p = \int_0^1 (f^2 + h_p^{-2}) \Phi_p \, \D p - 
[ \{3 r - 2 h (q, p)\} \Phi (p)]_{p=1} ,
\]
which must hold for every $\Phi \in H^{1}_0 (0,1)$; this space consists of
continuous functions that vanish at $p = 0$ and have derivatives in $L^2 (0, 1)$.
The last identity is well defined because the functions $f$ and $h_p^{-1}$ are
bounded when $(\psi, \eta)$ is a solution of problem $\rm P_r^M$ (see Proposition
\ref{PropBoundsPsi}).

This weak formulation allows us to use the technique developed in \cite{KMz1} (see
also \cite{KMz2}) for ordinary differential equations with operator coefficients. It
is worth mentioning that the explicit form of the operator $\cal A$ (the
differential expression and boundary operators) can be found in \cite{GW}.

\vspace{-2mm}

\subsubsection{Linearization Near a Stream Solution}

Let us linearize the equations \eqref{eq:v1} and \eqref{eq:v2} near the stream
solution $H(p, \lambda_+ (r))$. For this purpose we put $h = H + w$, thus obtaining
that the pair $(w, f)$ must satisfy the system
\begin{numcases}{}
w_q - f H_p = {\cal N}_1 (w, f) , \label{eq:weak1} \\
f_q - {\cal L} (w) = {\cal N}_2 (w, f) . \label{eq:weak2}
\end{numcases}
Here $(q, p) \in S$, ${\cal N}_1(w, f) = f w_p$, whereas the operators $\cal L$ and
${\cal N}_2$ are defined in the same way as $\cal A$. Their action on functions
given on the cross-section of $S$ is described by the following integral identities:
\begin{eqnarray*}
&& \int_0^1 {\cal L} (w) \Phi \, \D p = \int_0^1 \frac{w_p \Phi_p}{H_p^3} \, \D p -
[w \Phi]_{p=1} , \\ && \int_0^1 {\cal N}_2 (w) \Phi \, \D p = - \int_0^1 \left[
\frac{w_p^3}{H_p^3h_p^2} + \frac 3 2 \frac{w_p^2}{H_p^2h_p^2} + \frac{f^2}{2}
\right] \! \Phi_p \, \D p .
\end{eqnarray*}
They must be fulfilled for every $\Phi \in H^1_0(0, 1)$. Note that $\cal L$ is
nothing else than the operator of the spectral problem \eqref{SL1}--\eqref{SL3}
represented in a weak form.

Using Theorem \ref{propSmall} and Proposition \ref{RegProp}, we prove the following
assertion.

\begin{lemma} \label{lemma_small} 
For any $\epsilon, M > 0$ there exists $r^* \in (r_c, r_0)$ depending only on $M,
\epsilon$ and $\omega_1$ such that if $(\psi,\eta)$ is a solution of problem $\rm
P_r^M$ with $r \in (r_c,r^*]$, then the following inequality holds:
\[ \|w\|_{C^2(\bar S)} + \|f\|_{C^1(\bar S)} < \epsilon .
\]
\end{lemma}

\begin{proof}
Let the pair $(w,f)$ corresponds to a solution $(\psi, \eta)$ of problem ${\rm
P_r^M}$ with $r \in (r_c, r']$. Here $r' = r' (\epsilon_*, M, \omega_1) \in
(r_c, r_0)$ that exists for every $\epsilon_* > 0$ (it will be chosen later
depending on the values of $\epsilon$, $M$ and $\omega_1$) by assertion (a) of
Theorem \ref{propSmall}. Moreover, assertion (b) of Theorem \ref{propSmall}
guarantees that there exist $r'' = r''(M, \omega_1)$ and $\delta_* (M, \omega_1)
> 0$ such that both assertions of that theorem are true for any solution $(\psi,
\eta)$ of problem $\rm P_r^M$ with $r \in (r_c, \min \{r', r''\}]$. Without loss
of generality we suppose that condition \eqref{BdR} is also fulfilled for $r'$
and $r''$.

According to Proposition \ref{RegProp}, we have that $\psi \in C^{2,\alpha} (\bar
D)$ for some $\alpha \in (0, 1)$. For both functions $\psi$ and $\eta$ (the latter
bounds $D$ from above) their $C^{2,\alpha}$-norms are bounded by the same constant
depending on $M$ and $\omega_1$, but independent of $\epsilon_*$; that is,
\begin{equation} \label{C2a}
\| \psi \|_{C^{2,\alpha} (\bar D)} , \ \| \eta \|_{C^{2,\alpha} (\Bbb R)} \leq C (M,
\omega_1) .
\end{equation}

We recall that $w = h - H$ and $f = h_q / h_p$, where $H = H(p; \lambda_+(r))$ is a
stream solution used in the proof of Theorem \ref{propSmall} and $h$ corresponds to
$(\psi,\eta)$ through the partial hodograph transform. Since $w (q, 1) = \eta (q) -
d_+ (r)$, assertion (a) of Theorem \ref{propSmall} gives that $|w (q,1)| <
\epsilon_*$ for all $q \in \Bbb R$. Moreover, we have that $w (q,0) = 0$, and so the
maximum principle yields that $\| w \|_{L^{\infty} (\overline{S})} < \epsilon_*$.
Combining this and \eqref{C2a}, we conclude by virtue of interpolation argument
that
\[ \| w \|_{C^{2} (\overline{S})}, \ \  \| f \|_{C^1 (\overline{S})} < 
\epsilon_*^\gamma C (M, \omega_1) , \ \ i = 1,2 ,
\]
where $\gamma = \frac{\alpha}{2 (1+\alpha)}$. Now, choosing $\epsilon_*$ so that
$\epsilon_*^\gamma C (M, \omega_1) = \epsilon$, we complete the proof.
\end{proof}

\subsubsection{Spectral Splitting} 
\label{secSP}

The spectral problem related to the system \eqref{eq:weak1}, \eqref{eq:weak2} is as
follows:
\begin{numcases}{}
f H_p = \sigma w , \label{eq:spec1}\\ {\cal L} (w) = \sigma f ,  \label{eq:spec2}
\end{numcases}
where $p \in (0, 1)$. It is clear that $(f,w) \in L^2(0,1) \times H^1_0(0,1)$ is an
eigensolution corresponding to $\sigma$, if and only if $w = \phi$ and $f = \sigma
\phi H_p^{-1}$, where $\phi$ is an eigenfunction of the problem
\eqref{SL1}--\eqref{SL3} corresponding to the eigenvalue $\mu = \sigma^2$. Hence the
spectrum of \eqref{eq:spec1}--\eqref{eq:spec2} is the sequence $\{ \sigma_i \}_{i =
0}^{\infty}$ with $\sigma_i = \sqrt{\mu_i}$, where $\sigma_i$ is real and positive
for $i \geq 1$. The corresponding eigensolutions are $(w_i, f_i) = (\phi_i, \sigma_i
\phi_i H_p^{-1})$. Since $\mu_0$ is negative, there are two complex eigenvalues
$\sigma_0^{\pm} = \pm i \sqrt{|\mu_0|}$ and the corresponding two-dimensional
eigenspace is spanned by $(\phi_0, 0)$ and $(0, \phi_0 H_p^{-1})$.

Given a fixed $q \in \Bbb R$ and the real-valued functions $w (q,p)$ and $f (q,p)$,
then there are the following spectral decompositions:
\[ w (q,p) = \sum_{i=0}^{\infty} w_i (q) \phi_i (p) , \quad f(q,p) = 
\sum_{i=0}^{\infty} f_i (q) \phi_i (p) H_p^{-1} .
\]
In view of these formulae  we define two projectors ${\cal P} (w, f) = ({\cal P}_1
w, {\cal P}_2 f)$ and ${\cal Q} (w, f) = (w,f) - {\cal P} (w, f)$, where
\begin{equation}
{\cal P}_1 (w) = \phi_0 \int_0^1 w \phi_0 H_p^{-1} \D p \ \ \mbox{and} \ \ {\cal
P}_2 (f) = \phi_0 H_p^{-1} \int_0^1 f \phi_0 \, \D p . \label{calP}
\end{equation}
This leads to the following spectral splitting:
\[ (w , f) = {\cal P} (w, f) + {\cal Q} (w, f) .
\]
Here the first term is equal to $(w_0 \phi_0, f_0 \phi_0 H_p^{-1})$, whereas the
second one we denote by $(\widetilde{w}, \widetilde{\zeta})$. Applying the
projectors ${\cal Q}$ and ${\cal P}$ to the equations \eqref{eq:weak1}  and
\eqref{eq:weak2}, we obtain
\begin{numcases}{}
\widetilde{w}_q = \widetilde{\zeta} H_p + (I - {\cal P}_1) {\cal N}_1 (w,f),
\label{eq:wtilde} \\ \widetilde{\zeta}_q = {\cal L} (\widetilde{w}) + 
(I - {\cal P}_2) {\cal N}_2 (w,f)  \label{eq:ftilde}
\end{numcases}
for $(q, p) \in S$, and
\begin{numcases}{}
(w_0)_q = f_0 + \int_0^1 {\cal N}_1 (w,f) \phi_0 H_p^{-1} \D p , \label{eq:w0}\\
(f_0)_q = \mu_0 w_0 + \int_0^1 {\cal N}_2 (w,f) \phi_0 \, \D p  \label{eq:f0}
\end{numcases}
for $q \in \Bbb R$. Let us prove that ${\cal L}$ is a positive operator.

\begin{lemma} \label{LemmaSL} 
Let $r$ satisfy condition \eqref{BdR}. For all $w \in H^{1}_0 (0,1)$ orthogonal to
the function $\phi_0 H_p^{-1}$ in $L^{2} (0,1)$ the inequality
\[ \int_0^1 {\cal L} w \cdot w \, \D p \geq C \| w \|^2_{H^1 (0,1)}
\]
holds with a positive constant $C$ depending only on $\omega_0$.
\end{lemma}

\begin{proof} 
Using the spectral representation of $w$, one immediately obtains
\begin{equation} \label{l:3.1}
\int_0^1 {\cal L} w \cdot w \, \D p \geq \mu_1 \int_0^1 |w|^2 H_p^{-1} \D p ,
\end{equation}
where $\mu_1 > 0$ satisfies assertion (ii) of Proposition \ref{PropK1}. On the other
hand, the definition of $\cal L$ yields that
\[ \int_0^1 {\cal L} w \cdot w \, \D p  = \int_0^1 \frac {w_p^2}{H_p^3} \D p - w^2 (1) .
\]
Using the Schwarz inequality, we estimate $w^2 (1) = \int_0^1 w w_p \, \D p / 2$ by
\[ \frac 1 2 \left[ \int_0^1 \!\! w^2 H_p^3 \D p \right]^{1/2} \left[ 
\int_0^1 \!\! w_p^2 H_p^{-3} \D p \right]^{1/2} \leq \frac{\frak M^4}{2}
\int_0^1 \!\! w^2 H_p^{-1} \D p + \frac{1}{2} \int_0^1 \!\! w_p^2 H_p^{-3} \D p ,
\]
where $\frak M = \max_{[0,1]} H_p$. Then we obtain that
\[ \int_0^1 {\cal L} w \cdot w \, \D p  \geq \frac 1 2 \int_0^1 \frac {w_p^2}{H_p^3} 
\D p - \frac{\frak M^4}{2} \int_0^1 \frac {w^2}{H_p} \, \D p .
\]
Combining this inequality and \eqref{l:3.1}, we arrive at the required inequality.
\end{proof}

\subsubsection{Estimates for the Linearized Problem}

Let us consider the linearized system \eqref{eq:wtilde}--\eqref{eq:f0}. The linear
part of \eqref{eq:wtilde} and \eqref{eq:ftilde} is as follows:
\begin{numcases}{}
\widetilde{\xi}_q - \widetilde{\zeta} H_p = g^{(1)} , \label{eq:wtilde1} \\
\widetilde{\zeta}_q - {\cal L} (\widetilde{\xi}) = g^{(2)} , \label{eq:ftilde1}
\end{numcases}
where $(q, p) \in S$. In our considerations, it is sufficient to regard $g^{(1)}$
and $g^{(2)}$ as functions belonging to $C (\Bbb R; L^2 (0,1))$, whereas
$\widetilde{\xi}$ and $\widetilde{\zeta}$ are from $C^1 (\Bbb R; H^1_0 (0,1))$ and
$C^1 (\Bbb R; L^2 (0,1))$, respectively. Moreover, for every $q \in \Bbb R$ the
functions $\widetilde{\xi} (q, \cdot)$ and $\widetilde{\zeta} (q, \cdot)$ are
orthogonal in $L^2 (0,1)$ to $\phi_0 H_p^{-1}$ and $\phi_0$, respectively.
Furthermore, the linearized equations \eqref{eq:w0} and \eqref{eq:f0} are as
follows:
\begin{numcases}{}
\xi_0' - \zeta_0 = g^{(1)}_0 , \label{eq:w01} \\ \zeta_0' + |\mu_0| \xi_0 =
g^{(2)}_0 . \label{eq:f01}
\end{numcases}
Here $q \in \Bbb R$, $\xi_0, \zeta_0 \in C^1 (\Bbb R)$, $g^{(1)}_0, g^{(2)}_0 \in C
(\Bbb R)$ and $'$ denotes $\D / \D q$.

Our aim is to estimate solutions of these systems using the following norm:
\begin{equation}
\| w \|_{\theta, q_0}^2 = \int_S |w|^2 \E^{-\theta |q - q_0|} \, \D q \D p , \quad
\mbox{where} \ q_0 \in \Bbb R \ \mbox{and} \ \theta > 0 .
\label{norm}
\end{equation}

\begin{lemma} \label{StuctEst} 
Let $(\widetilde{\xi} , \widetilde{\zeta})$ be a solution of the system
\eqref{eq:wtilde1}, \eqref{eq:ftilde1}. If $\widetilde{\xi}$ and $\widetilde{\xi}_q$
belong to $L^\infty (\Bbb R; H^1_0 (0,1))$ and $\widetilde{\zeta}$ and
$\widetilde{\zeta}_q$ belong to $L^\infty (\Bbb R; L^2 (0,1))$, then there exists
$\theta_0 \in (0, 1/4)$ depending on $\omega_0$ such that for any $\theta \in (0,
\theta_0]$ and $q_0 \in \Bbb R$ the following inequality holds:
\begin{equation}
\| \widetilde{\zeta} \|_{\theta, q_0}^2 + \| \widetilde{\xi} \|_{\theta, q_0}^2 +
\| \widetilde{\xi}_p \|_{\theta, q_0}^2 \leq C \Big| \int_S \left[ g^{(1)}
\widetilde{\zeta} + g^{(2)} \widetilde{\xi} \right] \E^{-\theta |q - q_0|} \, \D q
\D p \Big| , \label{eq:Struct1}
\end{equation} 
Here the positive constant $C$ depends only on $\omega_0$.
\end{lemma}

Unlike usual estimates, in which solution's norm is estimated by a norm of the term
on the right-hand side, the inequality \eqref{eq:Struct1} has an expression
involving $g_1$ and $g_2$ on the right-hand side. The reason for this is the fact
that $g_1$ and $g_2$ are some integral functionals (see \eqref{eq:wtilde} and
\eqref{eq:ftilde}).

\begin{proof}
Let us multiply \eqref{eq:wtilde1} and \eqref{eq:ftilde1} by $\widetilde{\zeta}$ and
$\widetilde{\xi}$, respectively, and integrate the results over $(0,1)$. Summing up,
we obtain that
\[ \int_0^1 [\tilde{\xi} \tilde{\zeta}]_q \, \D p = \int_0^1 \left[ 
\widetilde{\zeta}^2 H_p + ({\cal L} \widetilde{\xi})  \widetilde{\xi} \right] \D p
+ \int_0^1 \left[ g^{(1)} \widetilde{\zeta} + g^{(2)} \widetilde{\xi} \right] \D p =
I_1 + I_2 .
\]
Applying Lemma \ref{LemmaSL}, we estimate $I_1$ from below as follows:
\[ I_1 \geq  \frak m \| \widetilde{\zeta} (q,\cdot) \|_{L^2(0,1)}^2 + C (\omega_0) 
\| \widetilde{\xi} (q,\cdot) \|_{H^1(0,1)}^2 ,
\]
where $C$ is the constant from Lemma \ref{LemmaSL}. Let $\theta_0 = \min \{ C
(\omega_0), \frak m\} / 2$, multiplying the last inequality by $\E^{-\theta |q -
q_0|}$ with $\theta \in (0, \theta_0)$ and $q_0 \in \Bbb R$, we obtain after
integration over $\Bbb R$
\[ \| \widetilde{\zeta} \|_{\theta, q_0}^2 + \| \widetilde{\xi} \|_{\theta, q_0}^2 + 
\| \widetilde{\xi}_p \|_{\theta, q_0}^2 \leq  \frac{1}{2 \theta_0} \left[ |I_2| +
\Big| \int_S \left( \tilde{\xi} \tilde{\zeta} \right)_q \E^{-\theta |q - q_0|} \,
\D q \D p \Big| \right] .
\]
Integrating by parts in the last integral, we see that its absolute value is less
that or equal to $\theta_0 \left( \| \widetilde{\zeta} \|_{\theta, q_0}^2 + \|
\widetilde{\xi} \|_{\theta, q_0}^2 \right)$, which leads to the following
inequality:
\[ \| \widetilde{\zeta} \|_{\theta, q_0}^2 + \| \widetilde{\xi} \|_{\theta, q_0}^2 + 
\| \widetilde{\xi}_p \|_{\theta, q_0}^2 \leq \frac{1}{\theta_0} \Big| \int_S \left[ g^{(1)}
\widetilde{\zeta} + g^{(2)} \widetilde{\xi} \right] \E^{-\theta |q - q_0|} \, \D q \D p \Big| .
\]
The proof is complete.
\end{proof}

\begin{lemma}\label{fmu} 
Let $(\xi_0, \zeta_0)$ be a bounded solution of the system \eqref{eq:w01},
\eqref{eq:f01} such that $\xi_0'$ and $\zeta_0'$ are also bounded. Then the
inequality
\begin{equation} 
\| \zeta_0 \|_{\theta, q_0}^2 \leq C_1(\omega_0) \Big[ \| \xi_0 \|_{\theta, q_0}^2 +
\| g^{(1)}_0 \|_{\theta, q_0}^2 + \| g^{(2)}_0 \|_{\theta, q_0}^2 \Big]
\label{eq:Struct2}
\end{equation}
holds for all $\theta \in (0,1/4]$, $q_0 \in \Bbb R$; the positive constant $C_1$
depends only on $\omega_0$.

Moreover, if $\xi_0 (q_0) = \zeta_0 (q_0) = 0$ for some $q_0 \in \Bbb R$, then 
\begin{equation} \label{eq:Struct3}
\| \xi_0 \|_{\theta, q_0}^2 + \| \zeta_0 \|_{\theta, q_0}^2 \leq C_2 (\omega_0, \theta) 
\| g_0^{(1)}  \|_{\theta, q_0}^2 + C_3 (\theta) \| g_0^{(2)} \|_{\theta, q_0}^2
\end{equation}
holds with an arbitrary $\theta > 0$; the positive constant $C_2$ depends on $\theta$ and $\omega_0$, while $C_3$ depends
only on $\theta$.
\end{lemma}

\begin{proof} 
Let us prove \eqref{eq:Struct2} first. Multiplying the equations \eqref{eq:w01} and
\eqref{eq:f01} by $\zeta_0$ and $\xi_0$, respectively, and summing up, we obtain
\[ (\xi_0 \zeta_0)' - \zeta_0^2 + |\mu_0| \xi_0^2 = g_0^{(1)} \zeta_0 + g_0^{(2)} \xi_0.
\]
This gives that
\[ \frac{3}{4} \zeta_0^2 \leq (\xi_0 \zeta_0)' + 4 \left[ g_0^{(1)}\right]^2 +  
\left[ g_0^{(2)}\right]^2 + (1 + |\mu_0|) \xi_0^2 .
\]
Let us multiply this inequality by $\E^{-\theta |q-q_0|}$ and integrate over $\Bbb
R$. This yields
\[ \frac{3}{4} \| \zeta_0 \|_{\theta, q_0}^2 \leq \int_{\Bbb R} (\xi_0 \zeta_0)' 
\E^{-\theta |q-q_0|} \, \D q + C \left[ \| g_0^{(1)} \|_{\theta, q_0} + \|
g_0^{(2)} \|_{\theta, q_0} + \| \xi \|_{\theta, q_0} \right].
\]
Estimating the integral on the right-hand side
\[ \int_{\Bbb R} (\xi_0 \zeta_0)' \E^{-\theta |q-q_0|} \, \D q \leq \theta \left(
\| \zeta_0 \|_{\theta, q_0}^2 + \| \xi_0 \|_{\theta, q_0}^2 \right) ,
\]
we see that for $\theta \leq \theta_0 = 1/4$ the required inequality follows from
the last two estimates.

Let turn to the second assertion. We multiply \eqref{eq:f01} by $\zeta_0$ and take
into account \eqref{eq:w01}, thus obtaining
\[ \zeta_0 \zeta_0' + |\mu_0| \xi_0 \xi_0' = |\mu_0| \xi_0 g_0^{(1)} + \zeta_0 g_0^{(2)} .
\]
Multiplying this by $\E^{-\theta |q-q_0|}$, we integrate the result over $(q_0,
+\infty)$ which gives
\[ \theta \int_{q_0}^{+\infty} \left[ \zeta_0^2 + |\mu_0| \xi_0^2 \right] \E^{-\theta 
|q-q_0|} \, \D q = \int_{q_0}^{+\infty} \left[ |\mu_0| \xi_0 g_0^{(1)} + \zeta_0
g_0^{(2)} \right] \E^{-\theta |q-q_0|} \, \D q
\]
after integration by parts. Applying the Schwarz inequality to the right-hand side,
we arrive at
\begin{equation} \label{fmu:1}
\frac{\theta}{2} \int_{q_0}^{+\infty} \left[ \zeta_0^2 + |\mu_0| \xi_0^2 \right] \E^{-\theta |q-q_0|} 
\, \D q \leq \frac{2|\mu_0|}{\theta} \| g_0^{(1)} \|_{\theta, q_0}^2 + \frac{2}{\theta} 
\| g_0^{(2)} \|_{\theta, q_0}^2 .
\end{equation}

Now, we multiply \eqref{eq:w01} by $\xi_0 \E^{-\theta |q-q_0|}$ and integrate the result 
over $(q_0, +\infty)$. This gives after integration by parts the following inequality:
\[ \theta \int_{q_0}^{+\infty} \xi_0^2 \E^{-\theta |q-q_0|} \, \D q = \int_{q_0}^{+\infty} 
\left[ \zeta_0 \xi_0 + \xi_0 g_0^{(1)} \right] \E^{-\theta |q-q_0|} \, \D q .
\]
Applying the Schwarz inequality to the right-hand side, we obtain
\[ \frac{\theta}{2} \int_{q_0}^{+\infty} \xi_0^2 \E^{-\theta |q-q_0|} \, \D q \leq 
\frac{4}{\theta} \left[ 1 + \frac{2|\mu_0|}{\theta} \right] \| g_0^{(1)} \|_{\theta, q_0}^2
+ \frac{8}{\theta^2} \| g_0^{(2)} \|_{\theta, q_0}^2 ,
\]
where \eqref{fmu:1} is also taken into account. Combining the last inequality and
\eqref{fmu:1}, we arrive at
\[
\int_{q_0}^{+\infty} \left[ \xi_0^2 + \zeta_0^2 \right] \E^{-\theta |q-q_0|} \, \D q \leq 
C_1 (\omega_0, \theta) \| g_0^{(1)} \|_{\theta, q_0}^2 + C_2 (\theta) 
\| g_0^{(2)} \|_{\theta, q_0}^2 .
\]

In the same way we estimate $\int_{-\infty}^{q_0} \left[ \xi_0^2 + \zeta_0^2 \right] 
\E^{-\theta |q-q_0|} \, \D q$, which completes the proof.
\end{proof}

\subsection{Proof of Theorem \ref{StabThm2'}}

Let $(\psi^{(1)}, \eta^{(1)})$ and $(\psi^{(2)}, \eta^{(2)})$ be two arbitrary
solutions of problem ${\rm P_r^M}$; here $M$ is a positive number and the Bernoulli
constant $r$ satisfies condition \eqref{BdR}.

For the functions $h^{(1)}$ and $h^{(2)}$ that correspond to $(\psi^{(1)},
\eta^{(1)})$ and $(\psi^{(2)}, \eta^{(2)})$, respectively, through the partial
hodograph transform, we put
\[ w^{(i)} = h^{(i)} - H \ \ \ \text{and} \ \ \ f^{(i)} = h^{(i)}_q / h^{(i)}_p , 
\quad i = 1, 2,
\]
where $H = H(p; \lambda_+(r))$. Let
\[ \xi = w^{(1)} - w^{(2)} \quad \mbox{and} \quad   \zeta = f^{(1)} - f^{(2)} ,
\]
for which we consider the following spectral splitting (see Section \ref{secSP} for
the corresponding notation):
\[ \xi = \xi_0 \phi_0 + \widetilde{\xi} \quad \mbox{and} \quad \zeta = \zeta_0 
\phi_0 H_p^{-1} + \widetilde{\zeta} .
\]
Here the functions $\widetilde{\xi} H_p^{-1}$ and $\widetilde{\zeta}$ are orthogonal 
to $\phi_0$ in $L^2(0,1)$. 

Let us outline our proof of Theorem \ref{StabThm2'}. First, we estimate the
$\theta, q_0$-norm (see \eqref{norm} for its definition) of $\widetilde{\xi}$ and
$\widetilde{\zeta}$ by the same norm of the one-dimensional projections $\xi_0$ and
$\zeta_0$ (see Lemma \ref{StabLemma} below). Next we estimate the norm of $\xi_0$ by
its Cauchy data (see Lemma \ref{CauchyLemma1}). The last step is to estimate the
Cauchy data of $\xi_0$ at some point by the Cauchy data of $\eta^{(1)} - \eta^{(2)}$
at the same point (see Lemma \ref{CauchyLemma2} below).

\subsubsection{Lemmas}

Assuming that both solutions $(w^{(i)}, f^{(i)})$, $i=1,2$ are small, we estimate
the difference of their projections by the norm of the function $\xi_0$ only.

\begin{lemma} \label{StabLemma} 
Let $r$ satisfy \eqref{BdR}. Then there exist $\theta_0, \epsilon_0$ and $C
(\omega_0)$ (all positive) such that the inequality
\[ \| \widetilde{\zeta} \|_{\theta, q_0}^2 + \| \widetilde{\xi} \|_{\theta, q_0}^2 
+ \| \widetilde{\xi}_p \|_{\theta, q_0}^2 \leq  \epsilon C (\omega_0) 
\| \xi_0 \|_{\theta, q_0}^2
\]
holds for all $\theta \in (0, \theta_0]$, $q_0 \in \Bbb R$ and $\epsilon \in
(0, \epsilon_0]$ provided
\[ \epsilon = \max_{i=1,2} \left( \|w^{(i)}\|_{C^2(\overline{S})} 
+ \|f^{(i)}\|_{C^1(\overline{S})} \right) \leq \epsilon_0.
\]
The constants $\epsilon_0$ and $\theta_0$ depend only on $\omega_0$ and the second
of them is the same as in Lemma \ref{StuctEst}, and so allows us to apply Lemma
\ref{fmu}.
\end{lemma}

\begin{proof} 
Since $\widetilde{\xi}$ and $\widetilde{\zeta}$ solve the system \eqref{eq:wtilde1},
\eqref{eq:ftilde1} with the following right-hand side terms (see \eqref{calP} for
the definition of ${\cal P}_i$):
\[ g^{(i)} = (I - {\cal P}_i) \left[ {\cal N}_i \left( w^{(1)}, f^{(1)} \right) - 
{\cal N}_i \left( w^{(2)}, f^{(2)} \right) \right] , \quad i = 1,2 ,
\]
Lemma \ref{StuctEst} implies that there exists $\theta_0 > 0$ such that the
inequality \eqref{eq:Struct1} holds for $\widetilde{\xi}$ and $\widetilde{\zeta}$
and all $\theta \in (0,\theta_0]$. Hence it remains to estimate the term on the
right-hand side of \eqref{eq:Struct1}; its integrand consists of two terms that have
$g^{(1)}$ and $g^{(2)}$ as factors.

The Schwarz inequality applied to the first term gives
\[ \left| \int_0^1 g^{(1)} \widetilde{\zeta} \, \D p \right| \leq  I \left[ \int_0^1 
\left\{ g^{(1)} \right\}^2 \, \D p \right]^{\frac 1 2} , \quad \mbox{where} \ [I
(q)]^2 = \int_0^1 \left( \widetilde{\xi}^2 + \widetilde{\xi}_p^2 +
\widetilde{\zeta}^2 \right) \D p .
\]
Inequality \eqref{BdR} together with Proposition \ref{PropK1} yield that the
operator ${\cal P}_1$ is bounded in $L^{2}(0,1)$ and its norm does not exceed a
constant depending only on $\omega_0$. Thus, we have
\[ \int_0^1 \left[ g^{(1)} \right]^2 \D p \leq C (\omega_0) \int_0^1 \left[ {\cal N}_1 
\left( w^{(1)}, f^{(1)} \right) - {\cal N}_1 \left( w^{(2)}, f^{(2)} \right) \right] 
\D p ,
\]
whereas the definition of ${\cal N}_1$ gives
\begin{eqnarray*}
\left| {\cal N}_1 \left( w^{(1)}, f^{(1)} \right) - {\cal N}_1 \left( w^{(2)},
f^{(2)} \right) \right| = \left| w^{(1)} f^{(1)} - w^{(2)} f^{(2)} \right| \\ \leq
\epsilon \, C(\omega_0) (I + I_0) , \ \ \ \mbox{where} \ [I_0 (q)]^2 = \xi_0^2 +
\zeta_0^2 .
\end{eqnarray*}
Therefore, 
\[ \left| \int_0^1 g^{(1)} \widetilde{\zeta} \, \D p \right| \leq  \epsilon \, C' 
(\omega_0) (I^2 + I_0^2) .
\]

Let us turn to estimating the second term on the right-hand side of
\eqref{eq:Struct1}:
\[ \int_0^1 g^{(2)} \widetilde{\xi} \, \D p = \int_0^1 \widetilde{\xi} (I - 
{\cal P}_2) \left[ {\cal N}_2 \left( w^{(1)}, f^{(1)} \right) - {\cal N}_2
\left( w^{(2)}, f^{(2)} \right) \right]  \, \D p = J_1 - J_2.
\]
Here $J_1$ stands for the integral whose integrand is th product of the square bracket
and $\widetilde{\xi}$, whereas the operator ${\cal P}_2$ is applied to the same
square bracket in $J_2$. Without loss of generality, we can assume that $\epsilon
\leq 1$. Then the definition of ${\cal N}_2$ gives that
\[ |J_1| \leq \epsilon \, C (\omega_0) (I^2 + I^2_0) .
\]
On the other hand, we have
\[ |J_2| \leq C (\omega_0) \int_0^1 |\widetilde{\xi}| \, \D p \cdot \Bigg| \int_0^1 
\left[ {\cal N}_2 \left( w^{(1)}, f^{(1)} \right) - {\cal N}_2 \left( w^{(2)},
f^{(2)} \right) \right] \phi_0 \, \D p \Bigg| .
\]
Using the definition of ${\cal N}_2$, we conclude that the last absolute value is
less than or equal to $\epsilon \, C (\omega_0) (I + I_0)$. Combining this fact and
the previous estimate, we get that
\[ |J_2| \leq \epsilon \, C (\omega_0) (I^2 + I_0^2) .
\]
Let $\theta \in (0,\theta_0]$, where $\theta_0$ is the same as in Lemma
\ref{StuctEst}. Then inequality \eqref{eq:Struct1} and the estimates obtained above
imply that
\begin{equation} \label{IbyI0}
\begin{split}
\int_{-\infty}^\infty [I(q)]^2 \E^{-\theta |q-q_0|} \, \D q & \leq \epsilon \, C
(\omega_0) \int_{-\infty}^\infty [I (q)]^2 \E^{-\theta |q-q_0|} \, \D q \\ & +
\epsilon \, C (\omega_0) \int_{-\infty}^\infty [I_0 (q)]^2 \E^{-\theta |q-q_0|} \,
\D q .
\end{split}
\end{equation}
In order to estimate the last term on the right-hand side, we apply inequality
\eqref{eq:Struct2}, thus obtaining
\[ \int_{-\infty}^\infty [I_0 (q)]^2 \E^{-\theta |q-q_0|} \leq C_1(\omega_0) \left[ 
\| \xi_0 \|^2_{\theta, q_0} + \| g_0^{(1)} \|^2_{\theta, q_0} + \| g_0^{(2)}
\|^2_{\theta, q_0} \right] ,
\]
where
\[ g_0^{(i)} = \int_0^1 \left[ {\cal N}_i \left( w^{(1)}, f^{(1)} \right) - 
{\cal N}_i \left( w^{(2)}, f^{(2)} \right) \right] \phi_0 \left[ H_p^{-1}
\right]^{|i-1|} dp, \quad i = 1,2.
\]
As above, one derives the following estimate:
\[ \| g_0^{(1)} \|^2_{\theta, q_0} + \| g_0^{(2)} \|^2_{\theta, q_0}
\leq \epsilon \, C'(\omega_0) (I^2 + I_0^2) .
\]
Therefore, if $\epsilon \leq [C_1 C']^{-1}/2$, we see that $\| I_0 \|_{\theta,
q_0}^2 \leq 2 C_1 (\omega_0) \| \xi_0 \|_{\theta, q_0}^2 +  I^2$. Combining this
inequality and \eqref{IbyI0}, we complete the proof by taking $\epsilon$ less than
or equal to $\epsilon_0 = \min \{ [C_1 C']^{-1}/2, C^{-1}/2 \}$.
\end{proof}

In the next lemma, we estimate the $\theta, q_0$-norm of $\xi_0$ by the Cauchy
data of this function at $q_0$.

\begin{lemma} \label{CauchyLemma1} 
Let $r$ satisfy \eqref{BdR}. Then there exist $\theta_0 > 0$ (that from Lemma
\ref{StabLemma} can be used) and $\epsilon_0 > 0$ (both depend only on $\omega_0$)
such that the inequality
\begin{equation} \label{eq:stab1}
\| \xi_0 \|_{\theta, q_0}^2 \leq C \left[ \xi_0 ^2 (q_0) + \xi_0'^2 (q_0) \right]
\end{equation}
holds for all $\theta \in (0, \theta_0]$ and $q_0 \in \Bbb R$ provided
\[ \max_{i=1,2} \left\{ \| w^{(i)} \|_{C^2(\overline{S})} + \| f^{(i)}
\|_{C^1(\overline{S})} \right\} \leq \epsilon_0 .
\]
The positive constant $C$ depends only on $\theta$ and $\omega_0$. 
\end{lemma}

\begin{proof} 
It follows from \eqref{eq:w0} and \eqref{eq:f0} that
\begin{equation} \label{eq:w02}
\xi_0'' + |\mu_0| \xi_0 = g , \quad q \in \Bbb R ,
\end{equation}
where
\begin{eqnarray*}
&& g = \int_0^1 \left[ {\cal N}_2 \left( w^{(1)}, f^{(1)} \right) - {\cal N}_2
\left( w^{(2)}, f^{(2)} \right) \right] \phi_0 \, \D p \\ && \ \ \ \ - \frac
{\partial}{\partial q} \int_0^1 \left[ {\cal N}_1 \left( w^{(1)}, f^{(1)} \right) -
{\cal N}_1 \left( w^{(2)}, f^{(2)} \right) \right] \phi_0 H_p^{-1} \, \D p .
\end{eqnarray*}
Then the solution $\xi_0$ of \eqref{eq:w02} has the following form:
\begin{eqnarray*}
&& \xi_0 (q) = \xi_0(q_0) \cos \sqrt{|\mu_0|} (q-q_0)  + \frac {\xi_0' (q_0)} {
\sqrt{|\mu_0|} } \sin \sqrt{|\mu_0|} (q-q_0) \nonumber \\ && \ \ \ \ \ \ \ + \frac
1 { |\sqrt{ \mu_0 }| } \int_{ q_0 }^{ q } g (q') \sin \sqrt { |\mu_0| } (q - q') \,
\D q' , \quad q_0 \in \Bbb R .
\end{eqnarray*}
In order to prove \eqref{eq:stab1} it is sufficient to estimate the $\theta,
q_0$-norm of the right-hand side in the last equality. The norm of the first two
terms is bounded by $C (\theta, \omega_0) \sqrt{\xi_0 ^2 (q_0) + \xi_0'^2 (q_0)}$.
Turning to the last term, we see that it is sufficient to prove the inequality
\begin{equation} \label{gest}
\frac 1 { |\sqrt{ \mu_0 }| } \left\| \int_{ q_0 }^{ q } g (q') \sin \sqrt {|\mu_0|}
(q - q') \, \D q' \right\|_{\theta, q_0} \leq \frac 1 2 \| \xi_0 \|_{\theta, q_0}.
\end{equation}
Let us apply Lemma \ref{fmu} (inequality \eqref{eq:Struct3}) to the function
\[ \frac 1 { |\sqrt{ \mu_0 }| } \int_{ q_0 }^{ q } g (q') \sin \sqrt { |\mu_0| } 
(q - q') \, \D q' .
\]
It is clear that this function satisfies the same equation \eqref{eq:w02} as
$\xi_0$, but with the homogeneous initial conditions at $q_0$. Therefore, Lemma
\ref{fmu} applied to this function and its derivative instead of $\xi_0$ and
$\zeta_0$, respectively, gives
\[ \frac 1 {\mu_0} \left\| \int_{ q_0 }^{ q } g (q') \sin \sqrt{|\mu_0|} (q - q') \, 
\D q' \right\|_{\theta, q_0}^2 \leq C (\omega_0, \theta) \| g \|^2_{\theta, q_0} .
\]
The last inequality is valid for all $\theta > 0$. Furthermore, according to the
definition of $g$, Lemma \ref{StabLemma} shows that the inequality
\[ \| g \|_{\theta, q_0}^2 \leq \epsilon \, C' (\omega_0, \theta) \| \xi_0 
\|_{\theta, q_0}^2
\]
holds for all $\theta \in (0, \theta_0]$ provided $\epsilon \in (0, \epsilon'_0]$
(here $\theta_0$ and $\epsilon'_0$ are the constants from Lemma \ref{StabLemma}),
for which purpose we take
\[ \epsilon = \max_{i=1,2} \left\{ \|w^{(i)}\|_{C^2(\overline{S})} +
\|f^{(i)}\|_{C^1(\overline{S})} \right\} .
\]
Finally, combining the last two inequalities, we arrive at \eqref{gest}, provided
\[ \epsilon \leq \epsilon_0 = \min \{ \epsilon_0', [C C']^{-1}/2 \} ,
\]
and this completes the proof.
\end{proof}

Now we estimate the Cauchy data of $\xi_0$ by the Cauchy data of $\xi$ at the same
point.

\begin{lemma} \label{CauchyLemma2} 
Let $r$ satisfy \eqref{BdR}. Then there exist positive constants $C$ and $\epsilon_0$ depending only on $\omega_0$
such that if $\max_{i=1,2} \left\{ \| w^{(i)} \|_{C^2(\overline{S})} + \| f^{(i)}
\|_{C^1(\overline{S})} \right\} \leq \epsilon_0$, then the inequality
\[ \xi_0 ^2 (q_0) + \xi_0'^2 (q_0) \leq C (\omega_0) 
\left[ \xi^2 (q_0,1) + \xi^2_q (q_0,1) \right]
\]
holds for all $q_0 \in \Bbb R$.
\end{lemma}

\begin{proof} 
It follows from the spectral splitting that
\begin{equation} \label{eq:W0rep}
\xi_0 (q_0) = \phi_0^{-1} (1) [ \xi (q_0,1) - \widetilde{\xi} (q_0,1) ] , \ 
\xi'_0 (q_0) = \phi_0^{-1} (1) [ \xi_q (q_0,1) - \widetilde{\xi}_q (q_0,1) ] .
\end{equation}
In order to estimate $\widetilde{\xi}(q_0,1)$ and $\widetilde{\xi}_q(q_0,1)$, we
write the system \eqref{eq:weak1}, \eqref{eq:weak2} in the following form:
\begin{eqnarray*}
&& \left[ \frac{w_p}{H_p^3} \right]_p + \left[ \frac{w_q}{H_p} \right]_q = [{\cal
N}_1^* (w, f)]_p + [{\cal N}_2^* (f)]_q , \quad (q, p) \in S , \\ &&
\frac{w_p}{H_p^3} - w = {\cal N}_1^* (w, f) \ \ \mbox{when} \ p = 1 , \quad w = 0 \
\ \mbox{when} \ p = 0 .
\end{eqnarray*}
Here
\[ {\cal N}_1^* (w, f) = \frac{w_p^3}{H_p^3 h_p^2} + \frac 3 2\frac{w_p^2}{H_p^2 h_p^2} 
+ \frac{f^2}{2} \ \ \mbox{and} \ \ {\cal N}_2^* (f) = \frac{f w_p}{H_p} .
\]
Comparing this and the system \eqref{eq:w0}, \eqref{eq:f0}, we obtain that
$\widetilde{w} = (I - {\cal P}_1)(w)$ must satisfy the problem
\begin{equation*}
\begin{split}
\left[ \frac{\widetilde{w}_p}{H_p^3} \right]_p + \left[ 
\frac{\widetilde{w}_q}{H_p} \right]_q & = [{\cal N}_1^*]_p + [{\cal N}_2^*]_q 
+ \frac{\phi_0}{H_p} \int_0^1 \left( {\cal N}_1^* \phi_0' - [{\cal N}_2^* \phi_0 ]_q \right) \, \D p  , \\
\frac{\widetilde{w}_p}{H_p^3} - 
\widetilde{w} & = {\cal N}_1^* \ \ \mbox{when} \ p = 1, \quad \widetilde{w} = 0 \ \ 
\mbox{when} \ p = 0 .
\end{split}
\end{equation*}

Let $\widetilde{w}^{(1)}$ and $\widetilde{w}^{(2)}$ correspond to $\eta^{(1)}$ and
$\eta^{(2)}$, respectively. Since both $\widetilde{w}^{(1)}$ and
$\widetilde{w}^{(2)}$ solve the last problem, but with different right-hand side
terms, their difference, which we also denote by $\widetilde{w}$, must satisfy the
following problem:
\begin{equation} \label{*}
\begin{split}
\left[ \frac{\widetilde{\xi}_p}{H_p^3} \right]_p + \left[
\frac{\widetilde{\xi}_q}{H_p} \right]_q & = [{\cal J}_1^*]_p + [{\cal J}_2^*]_q +
\frac{\phi_0}{H_p} \int_0^1 \left( {\cal J}_1^* \phi_0' - [{\cal J}_2^* \phi_0]_q
\right) \, \D p \\ \frac{\widetilde{\xi}_p}{H_p^3} - \widetilde{\xi} & = {\cal
J}_1^* \ \ \mbox{when} \ p = 1,  \quad \widetilde{\xi} = 0 \ \ \mbox{when} \ p = 0 .
\end{split}
\end{equation}
Here
\[ {\cal J}_1^* = {\cal N}_1^* \left( w^{(1)}, f^{(1)} \right) - {\cal N}_1^* 
\left( w^{(2)}, f^{(2)} \right) , \ \ {\cal J}_2^* = {\cal N}_2^* 
\left( w^{(1)}, f^{(1)} \right) - {\cal N}_2^* \left( w^{(2)}, f^{(2)} \right) .
\]
Similarly, the function $\xi$ must satisfy
\begin{equation} \label{**}
\begin{split}
& \left[ \frac{\xi_p}{H_p^3} \right]_p + \left[ \frac{\xi_q}{H_p} \right]_q  =
[{\cal J}_1^*]_p + [{\cal J}_2^*]_q, \\ & \frac{\xi_p}{H_p^3} - \xi = {\cal J}_1^* \
\ \mbox{when} \ p  = 1, \ \ \ \xi = 0 \ \ \mbox{when} \ p = 0.
\end{split}
\end{equation}

Let $\theta = \theta_0' = \theta_0''$ and $\epsilon_1 = \min \{
\epsilon_0',\epsilon_0'' \}$, where $\theta_0'$ $(\theta_0'')$ and $\epsilon_0'$
$(\epsilon_0'')$ are $\theta_0$ and $\epsilon_0$, respectively, that exist according
to Lemma \ref{StabLemma} (\ref{CauchyLemma1}, respectively). All of them depend
only on $\omega_0$. Without loss of generality, we assume that
\[ \epsilon = \max_{i=1,2} \left\{ \|w^{(i)}\|_{C^2(\overline{S})} +
\|f^{(i)}\|_{C^1(\overline{S})} \right\} \leq \min \{ 1,\epsilon_1 \} .
\]
First, let us show that
\begin{equation} \label{pf:T2:3}
\int_{\Bbb R} \|\xi\|^2_{C^{1,\alpha}(D_t)} \E^{-\theta |t - q_0|} \, \D t \leq C
(\omega_0) \int_{\Bbb R} \|\xi\|^2_{L^2 (2 D_t) } \E^{-\theta |t - q_0|} \, \D t ,
\end{equation}
where $D_t = [t-1, t+1] \times [0,1]$, $2 D_t = [t-2, t + 2] \times [0,1]$ and
$\alpha = 1/2$ (the latter can be any number between 0 and 1). Let us apply Theorem
9.3, \cite{ADN}, to the systems \eqref{*} and \eqref{**}. In the case of \eqref{**},
this theorem gives
\[ \|\xi\|_{C^{1,\alpha} (D_t)} \leq C (\omega_0) \left[ \| {\cal J}_1 \|_{C^{\alpha}
(2 D_{t})} + \| {\cal J}_2 \|_{C^{\alpha} (2 D_{t})} + \| \xi \|_{L^2 (2 D_t) }
\right] .
\]
It follows from the definition of ${\cal J}_{i}^*$, $i=1,2$, that
\begin{equation} \label{eq:ADN_W}
\| \xi \|_{C^{1,\alpha} (D_t)} \leq C (\omega_0) \left[ \epsilon \, \| \xi
\|_{C^{1,\alpha} (2 D_{t})} + \| \xi \|_{L^2 (2 D_t) } \right] .
\end{equation}

After squaring \eqref{eq:ADN_W} and multiplying the result by $\E^{-\theta
|t-q_0|}$, we integrate over $\Bbb R$, thus obtaining
\begin{eqnarray}
\int_{\Bbb R} \| \xi \|^2_{C^{1, \alpha} (D_t)} \E^{-\theta |t - q_0|} \, \D t
\leq \epsilon \, C (\omega_0) \int_{\Bbb R} \| \xi \|^2_{C^{1, \alpha} (2 D_t)}
\E^{-\theta |t - q_0|} \, \D t \nonumber \\ + \, C (\omega_0) \int_{\Bbb R} \| \xi
\|^2_{L^2 (2 D_t)} \E^{-\theta |t - q_0|} \, \D t = I_1 + I_2 .
\label{pf:T2:1}
\end{eqnarray}
Furthermore, we have that
\begin{equation} \label{pf:T2:2}
I_1 \leq C' (\omega_0) \epsilon \int_{\Bbb R} \| \xi \|^2_{C^{1,\alpha} (D_t)}
\E^{-\theta |t - q_0|} \, \D t ,
\end{equation}
which is a consequence of the inequality
\[ \| \xi \|^2_{C^{1, \alpha} (2 D_t)} \leq 2 \left[ \| \xi \|^2_{C^{1, \alpha}
(D_{t-1})} + \| \xi \|^2_{C^{1, \alpha} (D_{t})} + \| \xi \|^2_{C^{1, \alpha}
(D_{t+1})} \right] .
\]
Let $\epsilon < \epsilon_2 = [C']^{-1}/2$, where $C'$ is the constant in
\eqref{pf:T2:2}. Then \eqref{pf:T2:1} and \eqref{pf:T2:2} imply \eqref{pf:T2:3}.

On the other hand, applying Theorem 9.3, \cite{ADN}, to \eqref{*} and using the
definition of ${\cal J}_i^*$, $i=1,2$, we obtain
\begin{equation} \label{pf:T2:4}
\| \widetilde{\xi} \|_{C^{1, \alpha} (D_t)} \leq C (\omega_0) \left[ \epsilon \| \xi
\|_{C^{1, \alpha} (2 D_{t})} + \| \widetilde{\xi} \|_{L^2 (2 D_t) } \right] .
\end{equation}
Again we square this, multiply by $\E^{-\theta |t-q_0|}$ and integrate the result
over $\Bbb R$, thus obtaining
\begin{eqnarray*}
\int_{\Bbb R} \| \widetilde{\xi} \|^2_{C^{1, \alpha} (D_t)} \E^{-\theta |t - q_0|}
\, \D t \leq \epsilon \, C (\omega_0) \int_{\Bbb R} \| \xi \|^2_{C^{1, \alpha} (2
D_t) } \E^{-\theta |t - q_0|} \, \D t \\ + C (\omega_0) \int_{\Bbb R} \|
\widetilde{\xi} \|^2_{L^2 (2 D_t)} \E^{-\theta |t - q_0|} \, \D t .
\end{eqnarray*}
It follows from \eqref{pf:T2:2} and \eqref{pf:T2:3} that the right-hand side is less
than or equal to
\[ C (\omega_0) \epsilon \int_{\Bbb R} \| \xi \|^2_{L^2 (2 D_t)} 
\E^{-\theta |t - q_0|} \, \D t + C (\omega_0) \int_{\Bbb R} \| \widetilde{\xi}
\|^2_{L^2 (2 D_t)} \E^{-\theta |t - q_0|} \, \D t .
\]
Changing the order of integration, we apply Lemmas \ref{StabLemma} and
\ref{CauchyLemma1}, thus estimating these integrals as follows:
\[ \int_{\Bbb R} \| \widetilde{\xi} \|^2_{C^{1, \alpha} (D_t)} \E^{-\theta |t - q_0|}
\, \D t \leq \epsilon C'' (\omega_0) \left[ \xi_0^2 (q_0) + \xi_0'^2 (q_0) \right] .
\]
Combining this inequality and \eqref{eq:W0rep}, we obtain
\[ \xi_0 ^2 (q_0) + \xi_0'^2 (q_0) \leq C (\omega_0) [\xi(q_0,1)^2 + \xi'(q_0,1)^2] 
+ \epsilon C''' (\omega_0) [\xi_0 ^2 (q_0) + \xi_0'^2 (q_0)] .
\]
If $\epsilon < \epsilon_0 = \min \{ \epsilon_1, \epsilon_2, [2 C'''
(\omega_1)]^{-1/\gamma})$, then we get the required inequality.
\end{proof}

\subsubsection{Proof of Theorem \ref{StabThm2'}}

Now we are in a position to complete the proof of Theorem \ref{StabThm2'}. Let us
denote by $\epsilon_0'$, $\epsilon_0''$ and $\epsilon_0'''$ the constant
$\epsilon_0$ existing according to Lemma \ref{StabLemma}, \ref{CauchyLemma1} and
\ref{CauchyLemma2}, respectively. Let $\epsilon_0 = \min \{ \epsilon_0',
\epsilon_0'', \epsilon_0''' \}$, whereas $\theta_0$ is the constant existing by
Lemmas \ref{StabLemma} and \ref{CauchyLemma1}; both $\epsilon_0$ and $\theta_0$
depend only on $\omega_0$. Then assertion (a) of Theorem \ref{propSmall} and Lemma
\ref{lemma_small} allow us to find $r''$ depending on $M$ and $\omega_1$ so that the
inequality
\[ \max_{i=1,2} \left\{ \| w^{(i)} \|_{C^2(\overline{S})} + \| f^{(i)}
\|_{C^1(\overline{S})} \right\} \leq \epsilon_0
\]
holds for every solution of problem ${\rm P_r^M}$ provided $r \in (r_c, r'']$, where
$r''$ exists by Theorem \ref{StabThm2'}. Then we apply Lemmas \ref{StabLemma},
\ref{CauchyLemma1} and \ref{CauchyLemma2}, thus obtaining the inequality
\begin{eqnarray*}
&& \!\!\!\!\!\!\!\!\!\! \| w^{(1)}_p - w^{(2)}_p \|_{\theta, q_0} + \| f^{(1)} -
f^{(2)} \|_{\theta, q_0} \\ && \leq C (M,\omega_1) \left[ |w^{(1)} (q_0,1) - w^{(2)}
(q_0,1)|^2 + |w^{(1)}_q (q_0,1) - w^{(2)}_q(q_0,1)|^2 \right] ,
\end{eqnarray*}
from which \eqref{eq:stab'} follows. Thus, the proof of Theorem \ref{StabThm2'} is
complete.

\section{Proof of Theorem \ref{UniqThm} and Verification \\
of the Benjamin--Lighthill Conjecture}

Theorem \ref{UniqThm} is a consequence of the following three facts. First, in view
of Lemma \ref{lemma_small} all solutions of the problem $\rm P_r^M$ are of small
amplitude when $r$ belongs to $(r_c, r_*]$, where $r_*$ depends only on $M$ and
$\omega_1$. Second, according to results obtained in \cite{GW} all small amplitude
waves are exhausted by a continuous branch of Stokes waves bifurcating from a
horizontal shear flow and terminating by the solitary wave of elevation. Third,
Theorem \ref{StabThm} implies that these solutions are uniquely parametrized by
their height at the crest provided the latter lies on the $y$-axis.

Now we turn to verification of the Benjamin--Lighthill conjecture for nearcritical
values of Bernoulli's constant. Namely, we suppose that $r \in (r_c, r_*]$, where
$r_* < \min \{r', r''\}$ and $r', r''  \in (r_c, r_0)$ are the values that exist
according to Theorems \ref{UniqThm} and \ref{StabThm}, respectively.

Theorem \ref{UniqThm} says that Stokes-wave solutions of problem $\rm P_r^M$ are
parameterized by their heights at the crest provided the latter is located on the
$y$-axis. Let $(\psi^{(t)}, \eta^{(t)})$ be such a solution for some $t \in (d_+(r),
\eta^{(s)} (0))$ (we recall that $\eta^{(s)} (0)$ is the height of the corresponding
solitary wave at its crest). Since the flow force does not depend on $x$, its value
for $(\psi^{(t)}, \eta^{(t)})$ is as follows:
\[ s(t) = \left[ r + \frac 2 3 \Omega(1) \right] t - \frac 1 3 \left\{ t^2 + 
\int_0^t \!\! \left[ [\psi^{(t)}_x]_{x=0}^2 - [\psi^{(t)}_y]_{x=0}^2 + 2 \, \Omega
(\psi^{(t)} (0, y)) \right] \D y \right\} .
\]
Let us show that {\it this function strictly decreases on} $(d_+(r), \eta^{(s)} (0))$. 

For this purpose we write $s$ in terms of the function $h(q, p; t)$ that
corresponds to $(\psi^{(t)}, \eta^{(t)})$ through the partial hodograph transform,
thus obtaining
\[ 3 s (t) = \left[ 3 r + 2 \, \Omega (1) \right] t - t^2 + \int_0^1 \left[ 
\frac{1}{h_p^2 (0, p; t)} - 2 \, \Omega (p) \right] h_p (0, p; t) \, \D p .
\]
Denoting differentiation with respect to $t$ by the top dot, we get from the pre\-vious
equality that
\[ 3 \dot{s} (t) = 3r + 2 \, \Omega(1) - 2 t - \int_0^1 \left[ \frac {\dot{h}_p}{h_p^2} 
+ 2 \, \Omega (p) \dot{h}_p \right]_{q=0} \D p .
\]
Integrating by parts, we obtain
\[ 3 \dot{s} (t) = 3r + 2 \, \Omega (1) - 2 t - \dot{h} (0,1; t) \left[ \frac 1 
{h_p^2 (0,1; t)} + 2 \, \Omega (p) \right] + 2 \int_0^1 \left[ \frac{h_{qq}}{h_p}
\dot{h} \right]_{q=0} \D p .
\]
Here the equation \eqref{eq:hdf1} is taken into account to simplify the integrand.
Since $h(0,1; t) = t$, the Bernoulli equation \eqref{eq:hdf3} gives that
\[ \dot{s} (t) = \frac 3 2 \int_0^1 \left[ \frac{h_{qq}}{h_p} \dot{h} \right]_{q=0} \D p .
\]
Indeed, the out of integral terms cancel because the point $(q, p) = (0, 1)$
corresponds to a wave crest.

To complete the proof of our assertion, let us show that $h_{qq} (0,p; t) < 0$ and
$\dot{h} (0,p; t) > 0$ for all $p \in (0,1)$ and $t \in (d_+(r), \eta^{(s)}(0)]$.

In order to prove the first of these inequalities, we denote by $2\Lambda$ the 
wavelength of the Stokes wave described by $h(q,p; t)$ with some $t \in (d_+(r), 
\eta^{(s)} (0))$. Then we have that
\[ h_q (0,p) = h_q (\Lambda, p) = 0 \ \ \mbox{for all} \ p \in [0,1] \quad 
\mbox{and} \quad h_q(q,0) = 0 \ \ \mbox{for all} \ q \in [0,\Lambda] .
\]
On the other hand, the inequality $h_q(q,1) < 0$ holds for all $q \in (0,\Lambda)$. 
These properties of $h_q$ allow us to apply the maximum principle to this function
in the rectangle $(0,\Lambda) \times (0,1)$, which yields the required inequality
for $h_{qq} (0,p; t)$. 

Instead of proving the inequality $\dot{h} (0,p; t) > 0$, let us show that $h (0,p; t)$
is an increasing function of $t$ on $(d_+(r), \eta^{(s)}(0)]$. Putting 
\[ \xi (q,p) = h (q,p; t_1) - h(q,p; t_2) , \quad t_1, t_2 \in (d_+(r), \eta^{(s)}(0)] ,
\]
we combine the inequality \eqref{pf:T2:3} and Lemmas \ref{StabLemma},
\ref{CauchyLemma1}, \ref{CauchyLemma2}, where $q_0 = 0$, which gives
\[ \| \xi \|_{C^{1,\alpha} ([-1,1] \times [0,1])} \leq C (\omega_0) |t_1 - t_2| .
\]
This and \eqref{pf:T2:4} imply that
\[ \| \widetilde{\xi} \|_{C^{1,\alpha} ([-1,1] \times [0,1])} \leq \epsilon
\, C (\omega_0) |t_1 - t_2| ,
\]
where
\[ \epsilon = \max_{i=1,2} \left\{ \|w^{(i)}\|_{C^2(\overline{S})} +
\|f^{(i)}\|_{C^1(\overline{S})} \right\} .
\] 
The last inequality yields that
\[ \Big| \widetilde{\xi} (0, p) - \frac{\widetilde{\xi} (0,1)}{\phi_0 (1)} \phi_0 
(p) \Big| \Big/ |t_2 - t_1| < \epsilon \, C (\omega_0).
\]
Finally, we can write
\[ \xi (0,p) = \xi_0 (0) \phi_0 (p) + \widetilde{\xi}(0,p) = \xi (0,1) 
\frac{\phi_0 (p)}{\phi_0 (1)} + \left[ \widetilde{\xi} (0,p) - \frac{\widetilde{\xi}
(0,1)}{\phi_0 (1)} \phi_0 (p) \right].
\]
Thus, if $\epsilon$ is small enough, then the obtained inequalities give that
\[ [h(0,p; t_1) - h (0,p; t_2)] / (t_1 - t_2) > \frac{\phi_0 (p)}{2\phi_0 (1)} 
\quad \mbox{for all} \ p \in (0,1) ,
\]
which completes the proof of our assertion and, consequently, verification of the
Benjamin--Lighthill conjecture. \\

\noindent {\bf Acknowledgements.} V.~K. and E.~L. were supported by the Swedish
Research Council (VR). N.~K. acknowledges the support from G.\,S.~Magnuson's
Foundation of the Royal Swedish Academy of Sciences and Link\"oping University.

\end{document}